\documentclass{article}
\usepackage{fullpage}
\usepackage[utf8]{inputenc}

\title{Cohomology of bimultiplicative local systems on unipotent groups}
\author{Prashant Arote and Tanmay Deshpande}
\date{}

\usepackage[colorlinks=true]{hyperref}

\makeatletter
\@ifpackageloaded{hyperref}%
  {\newcommand{\mylabel}[2]
    {\protected@write\@auxout{}{\string\newlabel{#1}{{#2}{\thepage}%
      {\@currentlabelname}{\@currentHref}{}}}}}%
  {\newcommand{\mylabel}[2]
    {\protected@write\@auxout{}{\string\newlabel{#1}{{#2}{\thepage}}}}}
\makeatother

\usepackage{amsmath}
\usepackage{tikz-cd}
\usepackage[all,cmtip]{xy}
\usepackage{graphicx}
\usepackage{amsmath}
\usepackage{wasysym}
\usepackage{amssymb}
\usepackage{amsthm}
\usepackage{commath}
\usepackage{mathtools}
\usepackage{regexpatch,fancyvrb,xparse}
\usepackage{tikz-cd}
\usepackage[utf8]{inputenc}
\usepackage[nottoc]{tocbibind}
\usepackage{rotating}

\newtheorem{theorem}{Theorem}[section]
\newtheorem{lemma}[theorem]{Lemma}
\newtheorem{corollary}[theorem]{Corollary}

\theoremstyle{definition}
\newtheorem{definition}[theorem]{Definition}

\newtheorem{remark}[theorem]{Remark}
\newtheorem{consequence}[theorem]{Consequence}

\newcommand{\End}{\operatorname{End}}
\newcommand{\Ext}{\operatorname{Ext}}
\newcommand{\Qp}{\mathbb{Q}_{p}}

\newcommand{\Qlcl}{\overline{\mathbb{Q}}_{\ell}}
\newcommand{\Zp}{\mathbb{Z}_{p}}
\newcommand{\Ga}{\mathbb{G}_{a}}
\newcommand{\Gap}{\mathbb{G}_{a,\perf}}
\newcommand{\cpu}{\mathfrak{cpu}_{k}^{\circ}}
\newcommand{\E}{\mathcal{E}}
\newcommand{\Hom}{\operatorname{Hom}}
\newcommand{\perf}{\operatorname{perf}}
\newcommand{\spec}{\operatorname{Spec}}
\newcommand{\CH}{\operatorname{CH}}

\renewcommand{\L}{\mathcal{L}}
\renewcommand{\l}{\ell}

\begin{document}

\maketitle
\begin{abstract}
Let $U_1, U_2$ be connected commutative unipotent algebraic groups defined over an algebraically closed field $k$ of characteristic $p>0$ and let $\L$ be a bimultiplicative $\Qlcl$-local system on $U_1\times U_2$. In this paper we will study the $\Qlcl$-cohomology  $H^*_c(U_1\times U_2,\L)$, which turns out to be supported in only one degree. We will construct a finite Heisenberg group $\Gamma$ which naturally acts on $H^*_c(U_1\times U_2,\L)$ as an irreducible representation.  We will give two explicit realizations of this cohomology and describe the relationship between these two realizations as a finite Fourier transform.
    
\end{abstract}
\section{Introduction} Let $k$ be an algebraically closed field of characteristic $p>0$ and let $\l\neq p$ be a prime number. All schemes and group schemes considered in this paper are assumed to be defined over $k$, unless stated otherwise. The goal of this paper is to describe the cohomology of bimultiplicative $\Qlcl$-local systems on products of connected commutative unipotent groups. Multiplicative and bimultiplicative local systems play an important role in the theory of character sheaves on algebraic groups.

Let $U$ be a connected unipotent group over $k$. Then a multiplicative $\Qlcl$-local system on $U$ is a geometric analogue of a multiplicative character and the set of multiplicative local systems on $U$ is parametrized by the Serre dual (see Section \ref{Duality}) $U^*$ which is a \emph{perfect} commutative unipotent group over $k$. If $\L$ is a multiplicative local system on $U$, then 
\begin{equation}\label{eq:cohofmultlocsys}
H^*_c(U,\L)=\begin{dcases*}
0 & if $\L\ncong \Qlcl$,\\
\Qlcl[-2\dim U](-\dim U) & if $\L=\Qlcl$,
\end{dcases*}
\end{equation} where $[\cdot]$ denotes a cohomological degree shift and $(\cdot)$ denotes a Tate twist. Throughout this paper, we will use the notion of Serre duality for unipotent groups, which is only well defined in the setting of perfect unipotent groups. Hence it will often be necessary to work in the set-up of perfect schemes and perfect groups schemes, which we recall briefly in Section \ref{Duality}. Also to define Serre duality, it will be more convenient to work with central extensions of unipotent groups by $\Qp/\Zp$ instead of multiplicative $\Qlcl$-local systems. By \cite{Mitya}, if we fix an injective character $\psi:\Qp/\Zp\hookrightarrow \Qlcl^\times$ then these two notions are equivalent for unipotent groups. 

Now let $U_{1},\ U_{2}$ be connected commutative unipotent groups and let $\mathcal{L}$ be a \emph{bimultiplicative $\Qlcl$-local system} (see Section \ref{Duality} for a precise definition) on $U_{1}\times U_{2}$. Roughly speaking, this means that $\L$ is a local system on $U_1\times U_2$ whose restriction to each $\{x\}\times U_2\cong U_2$ (resp. $U_1\times \{y\}\cong U_1$) is a multiplicative local system on $U_2$ (resp. $U_1$). If $U_{1}$ and $U_{2}$ are connected commutative  unipotent  groups, then bimultiplicative $\Qlcl$-local systems on $U_{1}\times U_{2}$  are equivalent to biextensions of $U_{1}\times U_{2}$ by $\Qp/\Zp$ and we will often work interchangeably between these two notions.  For definitions and more details see Section \ref{Duality}.

If $\L$ is a bimultiplicative local system on $U_1\times U_2$, let $K_1\leq U_1$ be the closed subgroup formed by all $x\in U_1$ such that the restriction of $\L$ to $\{x\}\times U_2$ is trivial. Similarly we can define the closed subgroup $K_2\leq U_2$.  The cohomology $H^*_c(U_1\times U_2,\L)$ can be computed in two ways as follows. Consider the cartesian square:
\[
\xymatrix{
  & U_{1}\times U_{2} \ar[dl]_{p_{1}} \ar[rd]^{p_{2}} \\
  U_{1}\ar[rd]_{q_{1}} && U_{2} \ar[dl]^{q_{2}} \\
  & \{pt\}
}
\]
In the derived category for the derived pushforward with compact supports we have canonical isomorphisms
$$(q_{1}\circ p_{1})_{!}(\mathcal{L})\cong H^{*}_{c}(U_{1}\times U_{2}, \mathcal{L}) \cong (q_{2}\circ p_{2})_{!}(\mathcal{L}).$$ These isomorphisms give us two realizations of the cohomology which are described in Section \ref{computation}. In particular, we will see that the cohomology is supported in only one degree and we get two different bases for the cohomology $H^*_c(U_1\times U_2,\L)$. These two bases are parametrized by the connected components of the subgroups $K_1$ and $K_2$ defined above. We will denote these two bases by $\{X^{*}_{b_{2}}: b_{2}\in \pi_{0}(K_{2})\}$ and $\{Y^{*}_{b_{1}}: b_{1}\in \pi_{0}(K_{1})\}$. In this paper, we describe the matrix which relates these two bases.

Sometimes we will need to consider bimultiplicative local systems $\L$ on $U_1\times U_2$ where $U_i$ are \emph{perfect} connected commutative unipotent groups. To define the two bases of $H^*_c(U_1\times U_2,\L)$ as above in this perfect setting, we will need to choose models\footnote{A model of a perfect connected unipotent group over $k$ is a connected unipotent group over $k$ whose perfectization is identified with the given perfect connected unipotent group.} of the perfect group schemes $U_1$ and $U_2$. If we choose different models, the two bases would get scaled by some integral powers of $p$ and hence the change of basis matrix would also get scaled by some integral power of $p$.

We will see that the bimultiplicative local system $\L$ on $U_1\times U_2$ has natural equivariant structures for the translation actions of $K_1$ and $K_2$ on $U_1\times U_2$, but that $\L$ is not equivariant for the action of $K_1\times K_2$, rather we can define a  Heisenberg group $G$ extending $K_1\times K_2$ such that $\L$ has a natural $G$-equivariant structure (see Section \ref{sec:biadditive}). Hence the cohomology group $H^*_c(U_1\times U_2,\L)$ comes equipped with a $G$-action and we will prove that $H^*_c(U_1\times U_2,\L)$ is irreducible as a representation of $G$.
The Heisenberg group $G$ above is defined in terms of a natural  biadditive pairing $K_1\times K_2\to \Qp/\Zp$ which gives rise to a pairing $B:\pi_0(K_1)\times \pi_0(K_2)\to \Qp/\Zp$ (see Section \ref{sec:biadditive}). We will prove that the change of basis matrix relating the two bases of $H^*_c(U_1\times U_2,\L)$ is given by the above  pairing (composed with our fixed inclusion $\psi:\Qp/\Zp\hookrightarrow \Qlcl^\times$) up to a scalar, in particular the pairing above must be a perfect pairing. We also determine this scalar factor, which is of the form $\pm{p^a}$ for some $a\in \mathbb{Z}$ depending upon the choice of models for $U_1,U_2$.

If $U_{1}$ and $U_{2}$ are connected commutative unipotent groups, then  biextensions of $U_{1}\times U_{2}$ by $\Qp/\Zp$ are equivalent to biextensions of their perfectizations, which are in turn equivalent to $\Hom(U_{1,\perf},{U_{2,\perf}}^{*})$
(also equivalent to $\Hom(U_{2,\perf},{U_{1,\perf}}^{*})$, see Lemma \ref{biext-Hom}).
Let $\mathcal{L}$ be a bimultiplicative local system on $U_{1}\times U_{2}$ and $f:U_{1,\perf}\rightarrow {U_{2,\perf}}^{*},\ f^{*}:U_{2,\perf}\rightarrow {U_{1,\perf}}^{*}$ be the corresponding homomorphisms. 
Then the subgroups $K_{1}$ and $K_{2}$ defined above are essentially same as the  kernels of the homomorphisms $f$ and $f^{*}$ respectively. More precisely, $K_1,K_2$ provide models for the kernels of $f,f^*$ respectively. In Section \ref{computation}, we will prove that $\mbox{dim}(U_{1})+\mbox{dim}(\ker(f^{*}))=\mbox{dim}(U_{2})+\mbox{dim}(\ker(f))$.
In this paper we will prove the following results:
\begin{theorem}\label{MAIN RESULT}
Let $U_{1}$ and $U_{2}$ be two connected commutative unipotent algebraic group schemes  over $k$ of dimension $d_{1},\ d_{2}$ respectively.  Let
 $\mathcal{L}$ be a bimultiplicative $\Qlcl$-local system on $U_{1}\times U_{2}$ with the corresponding homomorphisms $f:U_{1,\perf}\rightarrow {U_{2,\perf}}^{*},f^{*}:U_{2,\perf}\rightarrow {U_{1,\perf}}^{*}$. Let $k_1,k_2$ denote the dimensions of $\ker(f),\ker(f^*)$ respectively.  Let $d=d_1-k_1=d_2-k_2$ and let $D=d_1+k_2=d_2+k_1$. Then $H^i_c(U_1\times U_2,\L)=0$ for $i\neq 2D$ and $H^{2D}_c(U_1\times U_2,\L(D))$ has the two explicit bases $\{X^{*}_{b_{2}}: b_{2}\in \pi_{0}(K_{2})\}$ and $\{Y^{*}_{b_{1}}: b_{1}\in \pi_{0}(K_{1})\}$. The relationship between these two bases of $H_{c}^{2D}(U_{1}\times U_{2}, \mathcal{L}(D))$ is given by,
$$Y^{*}_{b_{1}}=\frac{(-1)^d}{p^{r}}\sum_{b_{2}\in \pi_{0}(\ker f^{*})}\psi(B(b_{1},b_{2}))X^{*}_{b_{2}}$$
$$X^{*}_{b_{2}}=\frac{(-1)^d}{p^{r'}}\sum_{b_{1}\in \pi_{0}(\ker f)}\psi(-B(b_{1},b_{2}))Y^{*}_{b_{1}}$$
where $\psi:\Qp/\Zp\hookrightarrow \Qlcl^\times$  is our fixed injective character, $B:\pi_0(K_1)\times \pi_0(K_2)\to \Qp/\Zp$ is the associated biadditive pairing and $r,r'$ are some integers that we describe later, such that $p^{r+r'}=|\pi_{0}(\ker f)|=|\pi_{0}(\ker f^{*})|$.
\end{theorem}
\begin{remark}\label{rk:constsrr'}
Let us denote the scalar factors $(-1)^dp^r, (-1)^dp^{r'}$ above by $m(U_{1},U_{2},\L),\ m(U_{2},U_{1},\tau^{*}(\L))$ respectively. We will describe these constants in Section \ref{sec:proof} and Remark \ref{rk:constants}.
\end{remark}
\begin{remark}
We can also formulate the above result for bimultiplicative local systems on \emph{perfect} connected unipotent groups $U_1,U_2$. However, the constants $\ m(U_{1},U_{2},\L),m(U_{2},U_{1},\tau^{*}(\L))$ and the bases $\{X^{*}_{b_{2}}: b_{2}\in \pi_{0}(K_{2})\}$, $\{Y^{*}_{b_{1}}: b_{1}\in \pi_{0}(K_{1})\}$ will depend on the choice of models for  $U_{1},U_{2}$. However, we see that the product  $ m(U_{1},U_{2},\L)\cdot m(U_{2},U_{1},\tau^{*}(\L))$ does not depend on the choice of models.
\end{remark}

\noindent As a corollary of the above result, we deduce the following which is also proved in (\cite[Prop. A.19]{Mitya}):
\begin{corollary}\label{nondegenerate} 
The pairing $B$ is non-degenerate.
\end{corollary}
\noindent Using the non-degeneracy of $B$ we will prove the following:
\begin{corollary}\label{irreducible}
The cohomology  $H_{c}^{2D}(U_{1}\times U_{2}, \mathcal{L}(D))$ is an irreducible representation of the finite Heisenberg group $\Gamma=\pi_{0}(G)$.
\end{corollary}

Bimultiplicative local systems play an important role in the theory of character sheaves on unipotent groups (see \cite{BD,BD:foundations}).  The structure of the modular categories that arise in this theory is determined by certain skew-symmetric bimultiplicative local systems and their attached metric groups (see \cite{Datta,Desh}). The motivation for the questions studied in this paper also comes from the theory of character sheaves on unipotent groups. In the case of skew-symmetric and symmetric bimultiplicative local systems, we get a more explicit result, see Corollary \ref{symmetric}.

\section*{Acknowledgments}
We are grateful to Vladimir Drinfeld and Takeshi Saito for very helpful correspondence.

\section{Serre duality, biextensions and bimultiplicative local systems}\label{Duality}
In this section, we briefly recall the notion of Serre duality for commutative unipotent groups. For more details we refer to \cite{Mitya}, \cite{Datta}.
 A scheme $\mathbf{X}$ in characteristic $p$ (i.e., $p$ annihilates the
 structure sheaf $\mathcal{O}_{X}$ of $\mathbf{X}$) is called \textit{perfect}
 if the morphism $\mathcal{O}_{X} \rightarrow\mathcal{O}_{X}$
given by $f \mapsto f^{p}$ on the
local sections of $\mathcal{O}_{X}$ is an isomorphism of sheaves.

Let $\mathfrak{Perf}_{k}$ be the full subcategory of $\mathfrak{Sch}_{k}$ formed  by perfect schemes over $k$. One knows that there is a right adjoint to the forgetful functor from $\mathfrak{Perf}_{k}\ \mbox{to}\ \mathfrak{Sch}_{k}$, which is known as the \textit{perfectization} functor and denoted by $X\mapsto   X_{\perf}$ which we briefly recall below (see \cite{Greenberg} for more). 
There is a canonical adjunction morphism $X_{\perf}\to X$, for any scheme $X$ over $k$.
\begin{definition} Let $X$ be a scheme over $k$ then the underlying topological space of $X_{\perf}$ is same as that of $X$ and the structure sheaf of $X_{\perf}$ is the inductive limit of $\mathcal{O}_{X}\xrightarrow{\Phi^{*}}\mathcal{O}_{X}\xrightarrow{\Phi^{*}}\cdots$. 
Equivalently, $X_{\perf}=\varprojlim(X\xrightarrow{\Phi}X\xrightarrow{\Phi}\cdots)$ where $\Phi:X\rightarrow X$ is the absolute Frobenius morphism defined by identity on the underlying topological space and $f\mapsto f^{p}$ on local sections of the structure sheaf $\mathcal{O}_{X}$.

Let $X$ be a perfect scheme over $k$ then a model for $X$ is a scheme $X_{1}$ over $k$ such that $X_{1,\perf}\cong X$ as a perfect scheme over $k$.
A perfect scheme of finite type over $k$ is defined to be a perfect scheme over $k$ which is isomorphic to the perfectization of a scheme of finite type over $k$.
\end{definition}
Let $X$ be a scheme over $k$. Then we will write $X^{(p)}$ for the scheme over $k$ obtained as the fiber product of the structure morphism $X\rightarrow \spec k$ and the absolute Frobenius  morphism $\Phi_{k}:\spec k\rightarrow\spec k$.
By universal property of the fiber product, the morphism $\Phi:X\rightarrow X$ and the structure morphism $X\rightarrow \spec k$ induces a morphism $\Phi_{X/k}:X\rightarrow X^{(p)}$ of schemes over $k$; it is called the \textit{relative Frobenius morphism}.
The relative Frobenius morphism $X\rightarrow X^{(p)}$ induces an isomorphism between perfectizations.

Let $X, Y$ be  schemes of finite type over $k$ and let $f:X_{\perf}\rightarrow Y_{\perf}$ be any morphism between their perfectizations.
 Let $U$ be an affine open neighbourhood in $Y$, then $f^{*}(\mathcal{O}_{Y}(U)\subseteq\Phi^{n}(\mathcal{O}_{X}(f^{-1}(U)))$ for some $n$.
 As $X$ and $Y$ are schemes of finite type, we can choose a sufficiently large $N$ and an affine cover of $Y$ such that $f^{*}(\mathcal{O}_{Y}(U)\subseteq\Phi^{N}(\mathcal{O}_{X}(f^{-1}(U)))$  for all $U$ in an affine cover of $Y$.

\begin{remark}\label{model}
Let $X,\ Y$ be schemes of finite type over $k$ and let $f:X_{\perf}\rightarrow Y_{\perf}$ be any morphism then $f=\Phi^{-N}_{Y/k}\circ f'_{\perf}$ for some $f':X\rightarrow Y^{(p^{N})}$ (follows from above discussion).
\end{remark}

\begin{definition}{(cf. \cite[A.8]{Mitya})} A perfect unipotent group over $k$ is a perfect group scheme over $k$ which is isomorphic to the perfectization of a unipotent algebraic group over $k$.
\end{definition}

The two basic examples of perfect unipotent groups over $k$ are the discrete
group $\mathbb{Z}/p\mathbb{Z}$ and the perfectization $\mathbb{G}_{a, \perf}$ of the additive group $\Ga$. 
If $k$ is algebraically closed then every connected perfect unipotent group over $k$ has a finite filtration by closed normal subgroups with successive subquotients isomorphic to $\mathbb{G}_{a, \perf}$.
\begin{remark}\label{model:subgroup}
Let $G$ be a smooth group scheme over $k$ and let $G_{\perf}$ be its perfectization. Let $H$ be a perfect closed subgroup scheme of $G_{\perf}$.
Let $H_{1}\leq G$ be the underlying reduced and hence smooth subgroup scheme corresponding to $H$, then  $H_{1,\perf}\cong H$ as a subgroup of $G_{\perf}$.
Therefore if we fix a smooth model for perfect group scheme then we get   a canonical  model for any closed perfect subgroup scheme.
\end{remark}

Let $\cpu$ denote the category of perfect  connected commutative  unipotent algebraic groups over $k$. Let $U$ be a commutative connected unipotent group scheme in $\cpu$.  We consider the additive group $\Qp/\Zp$ as a discrete (infinite) commutative group scheme over $k$. Consider the following contravariant functor from the category $\cpu$ to  the category of abelian groups:
\begin{equation}\label{eq:centext}
S\mapsto \Ext(U\times S, \mathbb{Q}_{p}/ \mathbb{Z}_{p})
\end{equation}
where $\Ext(A,B)$ denotes (for a group scheme $A$ and a commutative group scheme $B$) the group  of isomorphism classes of central extensions of $A$ by $B$:
$$0\rightarrow B\rightarrow C\rightarrow A\rightarrow 0.$$
The functor $\cpu\to \mathfrak{Ab}^{\operatorname{op}}$ defined above is representable by a group scheme
$U^{*}$ in $\cpu$ (cf.\cite[Theorem A.9]{Mitya}). The group scheme $U^{*}$ is known as the \textit{Serre dual} of $U$. Serre duality has the following properties (cf. \cite[1.1]{Datta}):
\begin{itemize}
\item[(i)] $U^{*}$ is a perfect connected commutative unipotent group scheme isogenous to $U$.
\item[(ii)] There is a canonical isomorphism  $U\cong U^{**}$.
\item[(iii)\label{dualses}] If $ 0 \rightarrow U'\rightarrow U \rightarrow U'' \rightarrow 0$ is an exact sequence of perfect connected commutative unipotent group schemes, then so is 
$ 0 \rightarrow {U''}^*\rightarrow U^{*} \rightarrow {U'}^* \rightarrow 0$.
\end{itemize}
In other words Serre duality defines an exact anti-involution of the exact category $\cpu$.

\begin{remark}
More generally, if $U$ is any (not necessarily commutative or perfect) connected unipotent group, the functor defined by (\ref{eq:centext}) is representable by an object $U^*\in \cpu$. We refer to \cite[Appendix A.10]{Mitya} for details.
\end{remark}
\begin{remark}(cf.\cite[Remark F.1]{BD})\label{ext}
The group $\Ext(U\times S, \Qp/\Zp)$ equals $\Ext(U\times S,p^{-n}\Zp/\Zp)$ for any $n\in \mathbb{N}$ such
that $U$ is annihilated by $p^{n}$.
\end{remark}
\begin{remark}
Throughout this paper, we will fix an injective character $\Qp/\Zp\hookrightarrow \Qlcl^\times$. Then by \cite[Lemma 7.3]{Mitya} there is an equivalence between the notions of central extensions of a connected unipotent group $U$ by $\Qp/\Zp$ and multiplicative $\Qlcl$-local systems on $U$. Hence we may also consider the Serre dual $U^*$ as the moduli space of multiplicative local systems on the connected unipotent group $U$.
\end{remark}

\definition{(Biextensions, cf. \cite[A.6]{Mitya})}\label{biextension}
Let $G_1,G_2$ be group schemes over $k$ and let $A$ be a commutative group scheme over $k$. A biextension of $G_{1}\times G_{2}$ by $A$ is a scheme $E$ over $k$, equipped with an action of $A$ and a morphism $\pi : E \rightarrow G_{1}\times G_{2}$ which makes $E$ an $A$-torsor over $G_{1}\times G_{2}$, together with the following additional structures:
\begin{enumerate}
\item[(a)] Choices of sections of $\pi$ along $\{1\}\times G_{2}$ and $G_{1} \times \{1\}$, by means of which the
“slices” $\pi^{-1}({1}\times G_{2})$ and $\pi^{-1}(G_{1}\times \{1\})$ will be identified with $A\times G_{2}$ and
$G_{1}\times A$, respectively, where $1$ denotes the unit in $G_{1}$ or $G_{2}$.
\item[(b)] A morphism $\bullet_{1} : E\times_{G_{2}} E\rightarrow E$ which makes $E$ a group scheme over $G_{2}$ and
makes $\pi$ a central extension of $G_{1}\times G_{2}$, viewed as a group scheme over $G_{2}$, by
$A\times G_{2}$, in a way compatible with the identification $A\times G_{2}\cong\pi^{-1}({1}\times G_{2})$.
\item[(c)] A morphism $\bullet_{2} : E\times_{G_{1}} E\rightarrow E$ which makes $E$ a group scheme over $G_{1}$ and
makes $\pi$ a central extension of $G_{1}\times G_{2}$, viewed as a group scheme over $G_{1}$, by
$G_{1}\times A$, in a way compatible with the identification $G_{1}\times A\cong\pi^{-1}(G_{1}\times A)$.
\end{enumerate}
These data are required to satisfy the following compatibility condition:\\ if $T$ is
any $k$-scheme and $e_{11}, e_{12}, e_{21}, e_{22}\in E(T)=
\Hom_{k}(T,E)$, then
\begin{equation}\label{1}
    (e_{11}\bullet_{2} e_{12})\bullet_{1}(e_{21}\bullet_{2}e_{22})=(e_{11}\bullet_{1} e_{21})\bullet_{2}(e_{12}\bullet_{1}e_{22})
\end{equation}
whenever both sides of this equality are defined, i.e., whenever
$$\pi(e_{11})=(g_{1}, g_{2}), \ \pi(e_{12})=(g_{1}, g'_{2}),\
 \pi(e_{21})=(g'_{1}, g_{2}),\ \pi(e_{22})=(g'_{1}, g'_{2})\mbox{ for }g_{1}, g'_{1}\in G_{1}(T), g_{2}, g'_2\in G_{2}(T).$$

We will be primarily interested in biextensions of connected commutative unipotent groups $(U_1,U_2)$ by the discrete commutative group scheme $\Qp/\Zp$ or by a finite subgroup scheme $A\leq \Qp/\Zp$. 

\begin{remark}\label{finite biextension}
Let $E$ be a biextension of $U_{1}\times U_{2}$ by $\Qp/\Zp$ and $E_{0}$ be a connected component of $E$.
Let $A$ denote the stabilizer of the component $E_{0}$ in $\Qp/\Zp$ then $E_{0}$ is a biextension of $U_{1}\times U_{2}$ by $A$.
Therefore corresponding to every biextension of $U_{1}\times U_{2}$ by $\Qp/\Zp$, we get a biextension of $U_{1}\times U_{2}$ by a finite group $A\ (\leq \Qp/\Zp))$. 
In this paper we will work over a  biextension of $U_{1}\times U_{2}$ by a finite subgroup $A$ of $\Qp/\Zp$ corresponding to a biextension of $U_{1}\times U_{2}$ by $\Qp/\Zp$.
\end{remark}

For connected commutative unipotent groups, the notion of biextensions by $\Qp/\Zp$ is equivalent to the notion of bimultiplicative local systems which are our main objects of interest in this paper:
\definition{(Bimultiplicative local system, cf.  \cite[Appendix A.7.]{Mitya})}\label{bimultiplicative}
Let $G_{1}, G_{2}$ be group schemes over $k$, and let $\mu_{1} : G_{1}\times  G_{1} \rightarrow G_{1}  \mbox{ and }   \mu_{2} : G_{2} \times G_{2}\rightarrow G_{2}$
be the multiplication morphisms. Let
$$pr_{13}, pr_{23} : G_{1} \times G_{1} \times G_{2} \rightarrow G_{1} \times G_{2}
\mbox{ and }
pr_{12}, pr_{13} : G_{1} \times G_{2} \times G_{2}\rightarrow G_{1} \times G_{2}$$
be the projections. A bimultiplicative $\Qlcl$-local system on $G_{1} \times G_{2}$ is a $\Qlcl$-local system $\L$ on $G_1\times G_2$ such that
$(\mu_{1} \times id_{G_{2}})^{*}(\mathcal{L}) \cong pr_{13}^{*}(\mathcal{L})\otimes pr_{23}^{*}(\mathcal{L})$ as local systems on $ G_{1} \times G_{1} \times G_{2}$
and
$(id_{G_{1}} \times \mu_{2})^{*}(\mathcal{L})\cong
pr_{12}^{*}(\mathcal{L})\otimes pr_{13}^{*}(\mathcal{L})$ as local systems on $ G_{1} \times G_{2} \times G_{2}$, along with a trivialization $\L_{(1,1)}\cong\Qlcl$.

\remark\label{rk:bimultbiext} 
Recall that we have fixed an injection of the group $\mathbb{Q}_{p}/\mathbb{Z}_{p}$ into $\Qlcl^{\times}$. Let $U_1,U_2$ be connected commutative unipotent groups. In this case,  $\Qlcl$-bimultiplicative local systems on  $U_{1}\times U_{2}$ are in one to one  correspondence with   biextensions of $U_{1}\times U_{2}$ by $\Qp/\Zp$ (cf. \cite[Lemma 7.3 and Lemma A.16]{Mitya}).

For connected commutative unipotent groups $U_1,U_2$, let $\operatorname{Biext}(U_{1}, U_{2})$ denote the group of isomorphism classes of biextensions of $U_{1}\times U_{2}$ by the discrete group scheme $\mathbb{Q}_{p}/\mathbb{Z}_{p}$.
\lemma(cf. \cite[Lemma A.17]{Mitya})\label{biext-Hom} Let ${U_{1}}, {U_{2}} \in \cpu$ be perfect connected commutative unipotent groups.
There are canonical isomorphisms between abelian groups:
$$\Hom({U_{1}},{{U}_{2}}^{*})\cong \operatorname{Biext}({U_{1}}, {U_{2}})\cong \Hom({U}_{2},{U}_{1}^{*})$$
where the composite map is obtained from duality.
\begin{remark} (\cite[Lemma A.17]{Mitya}.) For every  $U \in \cpu$ there exist a biextension $\mathcal{E}_{U}$ of $U\times U^{*}$ corresponding to the homomorphism $Id:U\rightarrow U$, which is universal with respect to  following property:
Let $E$ be a biextension of $U\times U_{1}$ and let $f:U\rightarrow U_{1}^{*}$, $f^{*}:U_{1}\rightarrow U^{*}$ be the corresponding morphisms, then $E$ fits in the following pullback squares:
$$\begin{tikzcd}
E\arrow[d]\arrow[r]
& \mathcal{E}_{U_{1}^{*}}\arrow[d]\\
U\times U_{1}\arrow[r,"f\times Id"]
& U_{1}^{*}\times U_{1}
\end{tikzcd} 
\mbox{ and }
\begin{tikzcd}
E\arrow[d]\arrow[r]
& \mathcal{E}_{U}\arrow[d]\\
U\times U_{1}\arrow[r,"Id\times f^{*}"]
& U\times U^{*}.
\end{tikzcd}$$
\end{remark}

\definition(cf. \cite{Saibi}, Definition 1.5.1(ii).) A dual pair of (usual, i.e. non-perfect) unipotent groups is a triple $(U,U',\mathcal{E}_{U})$, where $U$ and $U'$ are (usual) connected commutative unipotent algebraic groups over $k$, and $\mathcal{E}_{U}$ is a biextension of $U\times U'$ by $\Qp/\Zp$ with the property that the induced biextension  
$(\mathcal{E}_{U})_{\perf}$ of $(U_{\perf} \times U'_{\perf})$ by $\Qp/\Zp$ is the universal family of central extensions of $U_{\perf}$ by
$\Qp/\Zp$ parameterized by $U'_{\perf}$ as in the remark above, and in particular  it
identifies $U'_{\perf}$ with 
$(U_{\perf})^{*}$ the Serre dual of $U_{\perf}$ . 

\remark\label{perfectization} (cf. \cite[Remark F.4(2)]{BD})
If $U$ is a connected commutative unipotent algebraic group over $k$, then there always exists a dual pair $(U,U',\mathcal{E}_U)$.
 Indeed, one can take $U'$ to be any commutative unipotent  group over
$k$ with $(U')_{\perf} \cong (U_{\perf})^{*}$, which exists 
because $(U_{\perf})^{*}$ is perfect  connected commutative  unipotent  algebraic group over $k$. Then we use a general
fact \cite[expos\'e VIII]{SGA4}:  if $X$ is a scheme over $k$ and $A$ is a discrete abelian
group, the natural functor from the groupoid of $A$-torsors over $X$ to the groupoid of
$A$-torsors over $X_{\perf}$ is an equivalence of categories.

\begin{remark} 
If $(U,U',\mathcal{E}_{U})$ is a dual pair then by definition we have an identification $\mathcal{E}_{U_{\perf}}\cong (\mathcal{E}_{U})_{\perf}$, where $\mathcal{E}_{U_{\perf}}$ is the universal family of central extensions of $U_{\perf}$ by $\Qp/\Zp$.
\end{remark}

\noindent $\bullet$ Now onwards $(U,U', \mathcal{E}_{U})$ will denote a dual pair, in particular we have ${U_{\perf}}^{*}\cong U'_{\perf}$.

In the remainder of this section we briefly recall the notion of Chow groups, Gysin maps and Poincare duality for etale cohomology. 
These concept are used in Section \ref{computation} to get the basis for cohomology.
For more detail we refer to \cite[Chap.VI Sec.5,9,10,11]{Milne}.

 \begin{remark}\label{Chow group}
 Let $X$ be a smooth scheme over k. An elementary $r$-cycle on $X$ is a closed integral subscheme $Z\subseteq X$ of codimension $r$.
 The group of algebraic $r$-cycles $C^{r}(X)$ is a free abelian group on the set of elementary cycles.
 The Chow group $\CH^{r}(X)$ is defined as the quotient $C^{r}(X)/\mbox{rational equivalence}$. 
 We write $C^{*}(X)=\bigoplus_{r\geq 0} C^{r}(X)$ and $\CH^{*}(X)=\bigoplus_{r\geq 0} \CH^{r}(X)$.
 \end{remark}
 \begin{remark}\label{Gysin Map}
 Let $X$ be a smooth scheme over k and $Z$ be a non-singular subscheme of $X$ such that each connected component of $Z$ has codimension $r$ in $X$ then there are canonical isomorphisms $H^{m-2r}(Z, \Qlcl(-r))\rightarrow H^{m}_{Z}(Z,\Qlcl)$  for all $m\geq 0$. 
 Using these isomorphisms and a long exact sequence of cohomology for a pair $(X,X\setminus Z)$, we get a homomorphism  $i_{*}:H^{0}(Z,\Qlcl)\rightarrow H^{2r}(X,\Qlcl(r))$.
 This homomorphism is called as \emph{Gysin} map.
 
 Let $X$ be a smooth scheme over $k$ the cycle class map $cl_{x}:C^{*}(X)\rightarrow H^{*}(X,\Qlcl)$ is a homomorphism of graded groups. For smooth cycles it is defined as follows:
 if $Z$ is a smooth closed integral subscheme of codimension $r$ then $cl_{X}(Z)=i_{*}(1_{Z})$ where $i_{*}:H^{0}(Z,\Qlcl)\rightarrow H^{2r}(X,\Qlcl(r))$ is the Gysin map and $1_{Z}\in H^{0}(Z,\Qlcl)$.
 The class map induces a ring homomorphism $[.]:\bigoplus\CH^{r}(X)\rightarrow\bigoplus H^{2r}(X,\Qlcl(r))$.
 \end{remark}
 \begin{remark}\label{Poincare Duality}
Let $X$ be a smooth scheme of dimension $d$ over $k$ and $\L$ be a $\Qlcl$-local system on $X$  with $\L^{\vee}$ being the dual local system. By the Poincare duality theorem, for each $r\in \mathbb{Z}$  we have a canonical perfect pairing $H^{r}_{c}(X,\L)\times H^{2d-r}(X,\L^{\vee}(d))\rightarrow H^{2d}_{c}(X,\Qlcl(d))\cong\Qlcl$.
 Using this one can identify $H^{2d-r}(X,\L^{\vee}(d))$ with the dual of $H^{r}_{c}(X,\L)$.
 \end{remark}
 \begin{lemma}\label{Pullback}
Let $f: X \rightarrow Y$ be a flat morphism of relative dimension $n$, and $\alpha$ a $r$-cycle on $Y$ which is rationally equivalent to zero. Then $f^{*}(\alpha)=[f^{-1}(\alpha)]$  is $r+n$-cycle on $X$ rationally equivalent to zero.
\end{lemma}
\begin{proof}
See \cite[Theorem $1.7$]{Intersection} .
\end{proof}
 \begin{remark}(cf.\cite[Prop. 9.2]{Milne})\label{9.2}
 Let $\pi:Y\rightarrow X$ be a map of varieties and $Z$ an algebraic cycle on $X$. 
 If for every prime cycle $Z'$ occurring in $Z$, $Y\times_{X}Z'$ is integral then $\pi^{*}(Z)$ is defined and $cl_{Y}(\pi^{*}(Z))=\pi^{*}cl_{X}(Z)$.
 \end{remark}
 \begin{remark}(cf.\cite[Prop.9.3]{Milne})\label{9.3}
 Let $i:Z\hookrightarrow X$ be a closed immersion of smooth varieties. For any $W\in C^{r}(Z),i_{*}(cl_{Z}(W))=cl_{X}(W)$, where $i_{*}$ is the Gysin map $H^{2r}(Z,\Qlcl(r))\rightarrow H^{2(r+c)}(X,\Qlcl(r+c))$ and $Z$ is a $c$-cycle.
 
 \end{remark}

\section{Biadditive pairing and Heisenberg group attached to biextensions}\label{sec:biadditive}

Let $\tilde{U}_{1},\ \tilde{U}_{2}$ be two perfect connected commutative unipotent groups.
Let $E$ be a biextension of $\tilde{U}_{1}\times \tilde{U}_{2}$ by $\Qp/\Zp$ and let $f : \tilde{U}_{1} \rightarrow \tilde{U}_{2}^{*}, f^{*} : \tilde{U}_{2}\rightarrow \tilde{U}_{1}^{*}$ be the corresponding homomorphisms. Let $\tilde{K}_1\subseteq \tilde{U}_1,\ \tilde{K}_2\subseteq \tilde{U}_2$ be the kernels of $f, f^*$ respectively.
Then the restriction $E\vert _{\tilde{K}_1 \times \tilde{U}_{2}}$ (resp. $E\vert _{\tilde{U}_{1} \times \tilde{K}_2}$) is a trivial biextension of $\tilde{K}_1 \times \tilde{U}_{2}$ (resp. $ \tilde{U}_{1}\times \tilde{K}_2$).
Since $\tilde{U}_{1}, \tilde{U}_{2}$ are connected, we get unique trivializations (cf. \cite[ A.13]{Mitya}):

$$\sigma_{1} : (\tilde{K}_1 \times \tilde{U}_{2})\times \mathbb{Q}_{p}/\mathbb{Z}_{p}\xrightarrow{\sim} E\vert _{\tilde{K}_1 \times \tilde{U}_{2}}$$  
$$\sigma_{2} :(\tilde{U}_{1} \times \tilde{K}_2)\times \mathbb{Q}_{p}/\mathbb{Z}_{p}\xrightarrow{\sim} E\vert _{\tilde{U}_{1} \times \tilde{K}_2}  $$
Both $\sigma_{1}\ \mbox{and}\ \sigma_{2}$ induce trivializations of $E\vert_{ \tilde{K}_1 \times \tilde{K}_2}$ by restriction, hence the composition $\sigma_{1}\circ\sigma_{2}^{-1}$ gives an automorphism of the biextension of 
 $\tilde{K}_1 \times \tilde{K}_2$.
Therefore we get a  biadditive map (of group schmes):
$$B: \tilde{K}_1 \times \tilde{K}_2 \rightarrow \mathbb{Q}_{p}/\mathbb{Z}_{p}$$
which is trivial on $\tilde{K}_1^{\circ} \times \tilde{K}_2^{\circ}$. Hence we get a biadditive pairing 
$$B : \pi_{0}(\tilde{K}_1)\times\pi_{0}(\tilde{K}_2)\rightarrow \mathbb{Q}_{p}/\mathbb{Z}_{p}$$
by passing to the quotient.
The groups of connected components $\pi_{0}(\tilde{K}_1)$ and $\pi_{0}(\tilde{K}_2)$ are finite abelian $p$-groups.
We will see later  that the pairing $B$ is non-degenerate.

By Remark \ref{finite biextension}, by replacing $E$ with a connected component, we consider $E$ as an extension of $\tilde{U}_{1}\times \tilde{U}_{2}$ by a finite subgroup $A\subseteq\Qp/\Zp$. We get the biadditive pairing 
$$B : \pi_{0}(\tilde{K}_1)\times\pi_{0}(\tilde{K}_2)\rightarrow A\subseteq\mathbb{Q}_{p}/\mathbb{Z}_{p}.$$
 Let $\tilde{G}$ be the perfect group scheme whose underlying perfect scheme is $\tilde{K}_{1}\times \tilde{K}_{2}\times A $, and where the group operation is defined by using the biadditive pairing $B$ as follows:
$$(b_{1}, b_{2}, a)\star (b'_{1}, b'_{2}, a')= \left( b_{ 1}+b'_{1},\ b_{2}+b'_{2},\ a+a'+B(b_{1}, b'_{2})\right).$$
Under this operation $\tilde{G}$ becomes a perfect Heisenberg group scheme with center $\tilde{K}_1^{\circ} \times \tilde{K}_2^{\circ}\times A$. We have a central extension $$0\to A\to \tilde{G}\to \tilde{K}_1\times \tilde{K}_2\to 0.$$  The Heisenberg group $\tilde{G}$ naturally acts (through its quotient $\tilde{K}_1\times \tilde{K}_2$) on $\tilde{U}_1\times \tilde{U}_2$. Let us now define a natural action of $\tilde{G}$ on $E$ such that the projection $\pi:E\to \tilde{U}_1\times \tilde{U}_2$ becomes $\tilde{G}$-equivariant. 
There are natural translation  actions of  $\tilde{K}_{1}$ and $\tilde{K}_{2}$  on $E$ through the  trivializations $\sigma_{1}$ and $\sigma_{2}$ as follows:  For $x \in E$, let $\pi(x)=(g_1,g_2)$ where $\pi:E\to \tilde{U}_1\times \tilde{U}_2$. Then
 $$k_{1}*x= \sigma_{1}(k_{1}, g_{2})\bullet _{1} x\mbox{ for }k_1\in \tilde{K}_{1}$$
 $$k_{2}*x= \sigma_{2}(g_{1}, k_{2})\bullet _{2} x\mbox{ for }k_2\in \tilde{K}_{2}$$
The two actions above do not commute, since
$$k_{1}*(k_{2}*x)= k_{1}*\left(\sigma_{2}(g_{1}, k_{2})\bullet_{2} x\right) =
\sigma_{1}(k_{1}, k_{2} + g_{2})\bullet_{1}\left(\sigma_{2}(g_{1}, k_{2})\bullet_{2} x\right) $$
$$=\left(\sigma_{1}(k_{1}, k_{2})\bullet_{2}\sigma_{1}(k_{1}, g_{2})\right)\bullet_{1}\left( \sigma_{2}(g_{1}, k_{2})\bullet_{2}x\right)$$
$$=\left(\sigma_{1}(k_{1}, k_{2})\bullet_{1}\sigma_{2}(g_{1}, k_{2})\right)\bullet_{2}\left( \sigma_{1}(k_{1}, g_{2})\bullet_{1}x\right)$$
 $$=\left( B(k_{1},k_{2}) \sigma_{2}(k_{1},k_{2})\bullet_{1}\sigma_{2}(g_{1},k_{2})\right)\bullet_{2}(k_{1}*x)$$
 $$= B(k_{1},k_{2})\sigma_{2}(k_{1}+g_{1}, g_{2})\bullet_{2}(k_{1}*x)$$
 $$=B(k_{1},k_{2})\left(k_{2}*(k_{1}*x)\right).$$
Therefore, we get  $k_{1}*(k_{2}*x)= B(k_{1},k_{2})\left( k_{2}*(k_{1}*x)\right)$.  Using this relation, we in fact get an action of  $\tilde{G}$ on $E$ as follows: for $x\in E$ and $(k_{1}, k_{2}, a) \in \tilde{G}$,
$$(k_{1}, k_{2}, a)*x = k_{1}*\left(k_{2}*\left(\left(a+B(k_{1}, k_{2})\right)\cdot x\right)\right).$$
It is clear that the projection $\pi:E\to \tilde{U}_1\times\tilde{U}_2$ is $\tilde{G}$-equivariant with this action.
The action of $\tilde{G}$ on $E$ gives rise to a representation of $\tilde{G}$ on  $H^{*}_{c}(E,\Qlcl)$. In fact, we get an action of the finite Heisenberg group $\pi_0(\tilde{G})=\pi_0(\tilde{K}_{1})\times \pi_0(\tilde{K}_{2})\times A$ on the cohomology $H^*_c(E,\Qlcl)$.

Now let $U_{1}, U_{2}$ be two commutative connected unipotent algebraic groups. Let $E$ be a biextension of $U_{1}\times U_{2}$ by a finite group $A\subseteq \Qp/\Zp$ and $f:U_{1,\perf}\rightarrow {U_{2,\perf}}^{*}$, $f^{*}:U_{2,\perf}\rightarrow {U_{1,\perf}}^{*}$ be the corresponding homomorphisms. We can choose models $U'_2$ for ${U_{2,\perf}}^{*}$ and $U'_1$ for ${U_{1,\perf}}^{*}$ such that $f,f^*$ are perfectizations of algebraic maps $f:U_1\to U'_2$ and $f^*:U_2\to U'_1$. These models for the duals give rise to dual pairs $(U_{1},U'_{1},\mathcal{E}_{U_{1}}),\ (U_{2},U'_{2},\mathcal{E}_{U_{2}})$ for $U_1$ and $U_2$. 

Let $K_1\subseteq U_1$ (resp. $K_2\subseteq U_2$) be the reduced subgroup scheme associated with $\ker f\subseteq U_1$ (resp. $\ker f^*\subseteq U_2$). As before, we have the biadditive pairing $B:\pi_0(K_1)\times \pi_0(K_2)\to A$ as well as the corresponding Heisenberg group $G:=K_1\times K_2\times A$. As before, we have a natural action of $G$ on $E$, which induces an action of the finite Heisenberg group $\pi_0(G)=\pi_0(K_1)\times\pi_0(K_2)\times A$ on $H^*_c(E,\Qlcl)$. Moreover, the natural projection $\pi:E\to U_1\times U_2$ is $G$-equivariant.
\begin{remark}\label{decomposition}
Since $\pi:E\to U_1\times U_2$ is a biextension by $A$, we have 
\begin{equation}
\pi_!\Qlcl=\bigoplus\limits_{\chi\in Irr(A)}\L_\chi,
\end{equation}
where for an irreducible character $\chi:A\to \Qlcl^\times$, $\L_\chi$ denotes the corresponding bimultiplicative local system on $U_1\times U_2$. Since $\pi$ is $G$-equivariant, each $\L_\chi$ is a $G$-equivariant local system on $U_1\times U_2$ and hence we have an action of $G$ on $H^*_c(U_1\times U_2,\L_\chi)$ such that $A\leq G$ acts by the character $\chi$. Hence we have a decomposition
$$H_{c}^{*}(E,\Qlcl)\cong \bigoplus_{\chi \in Irr(A)}H_{c}^{*}(U_{1}\times U_{2}, \L_{\chi})$$
as $G$-representations and the direct summand $H_{c}^{*}(U_{1}\times U_{2}, \L_{\chi})\subseteq H_{c}^{*}(E, \Qlcl)$ can be identified as the $\chi$-isotypic component for the $A$-action.
\end{remark}

 \section{Computation of the cohomology}\label{computation}
Let $U_{1},\ U_{2}$ be two commutative connected unipotent algebraic groups over $k$ of dimension $d_{1},d_{2}$ respectively.
Let $\L$ be a bimultiplicative $\Qlcl$-local system $U_{1}\times U_{2}$. By Remark \ref{rk:bimultbiext} this corresponds to a biextension of $U_{1}\times U_{2}$ by $\Qp/\Zp$.  
We will denote the connected component of above biextension by $\pi:E\to U_1\times U_2$. Then $E$ is a biextension of $U_{1}\times U_{2}$ by a finite group $A\leq \Qp/\Zp$ ($A$ is the subgroup which preserves the connected component $E$) and $\pi^*\L$ is a trivial local system.
We can choose models $U'_{1}$ for ${U_{1,\perf}}^{*}$ and $U'_{2}$ for ${U_{2,\perf}}^{*}$ in such a way that we have algebraic maps $f:U_{1}\rightarrow U'_{2}$, $f^{*}:U_{2}\rightarrow U'_{1}$ whose perfectizations  $f_{\perf}:U_{1,\perf}\rightarrow U^{*}_{2,\perf}$, $f_{\perf}^{*}:U_{2,\perf}\rightarrow U^{*}_{1,\perf}$ correspond to the bimultiplicative local system $\L$ on $U_{1}\times U_{2}$. 
The models for duals give rise to dual pairs $(U_{1},U'_{1},\E_{U_{1}})$, $(U_{2},U'_{2},\E_{U_{2}})$ for $U_{1}$ and $U_{2}$. We will fix these dual pairs for further calculations in this section. As before, we define $K_1\leq U_1, K_2\leq U_2$ to be the reduced subgroup schemes associated with $\ker f,\ \ker f^*$ respectively.

Consider the following cartesian square,  
\[
\xymatrix{
  &U_{1}\times U_{2} \ar[dl]_{p_{1}} \ar[rd]^{p_{2}} \\
  U_{1}\ar[rd]_{q_{1}} && U_{2} \ar[dl]^{q_{2}} \\
  & \{pt\}
}
\]
In the derived category, we have   $(q_{1}\circ p_{1})_{!}(\mathcal{L})\cong H^{*}_{c}(U_{1}\times U_{2}, \mathcal{L}) \cong (q_{2}\circ p_{2})_{!}(\mathcal{L})$. Let us do some computation: $H^{*}_{c}(U_{1}\times U_{2}, \mathcal{L})\cong (q_{1}\circ p_{1})_{!}(\mathcal{L})\cong (q_{1})_{!}\left( {p_{1}}_{!}(\mathcal{L})\right) \cong
 H_{c}^{*}(U_{1}, (p_{1})_{!}(\mathcal{L}))$
 but for any $x\in U_{1}$, the stalk of $(p_{1})_{!}(\mathcal{L})$ at $x$ is $(p_{1})_{!}(\mathcal{L})_x\cong H_{c}^{*}(x\times U_{2}, \mathcal{L}\vert_{x\times U_{2}})\cong H_{c}^{*}(U_{2}, \L_{x})$ (where $\L_{x}=\L\vert_{x\times U_{2}}$),
but we know that
\begin{align*}
    H_{c}^{*}(U_{2},\L_{x})
    &=\begin{cases}
			{\Qlcl}(-d_{2})[-2d_{2}], & \mbox{if}\ \L_{x}= \Qlcl \\
            0, & \mbox{otherwise}
		 \end{cases}\\
	& =\begin{cases}
\Qlcl(-d_{2})[-2d_{2}], & \mbox{if}\ x\in \ker f\\
0, & \mbox{otherwise.}
 \end{cases}
\end{align*}
Moreover, we have a canonical trivialization $\L|_{K_1\times U_2}\cong {\Qlcl}_{K_1\times U_2}$ (see beginning of Section \ref{sec:biadditive}). Hence $(p_1)_!\L={\Qlcl}_{K_1}(-d_2)[-2d_2]$.
Let dimension of $\ker f$ be $k_{1}$ and dimension of $\ker f^{*} $ be $k_{2}$. Now each connected component of $K_i$ is isomorphic to affine space $\mathbb{A}^{k_i}$. Hence   using the above calculation  we get that, $$H^*_c(U_1\times U_2,\L)\cong H_{c}^{*}(U_{1}, (p_{1})_{!}\mathcal{L})\cong\ \bigoplus_{b \in\pi_{0}(K_1)}\Qlcl(-d_{2}-k_{1})[-2d_{2}-2k_{1}].$$
Similarly, $$H^*_c(U_1\times U_2,\L)\cong H_{c}^{*}(U_{2}, (p_{2})_{!}\mathcal{L})\cong\ \bigoplus_{b \in\pi_{0}(K_2)}\Qlcl(-d_{1}-k_{2})[-2d_{1}-2k_{2}].$$
Here we already see that the cohomology is only supported in degree $2D:=2(d_2+k_1)$. Moreover, we get two realizations of the cohomology  $H_{c}^{2D}(U_{1}\times U_{2},\L(D))$ which determine two bases of this cohomology space. Now our aim is to determine the relation between these bases. We also observe that:
\consequence\label{dim eq}
\begin{enumerate}
\item  $D=d_{1}+k_{2}=d_{2}+k_{1}$. Also set $d:=d_2-k_2=d_1-k_1$.
\item $\vert\pi_{0}(K_1)\vert=\vert\pi_{0}(\ker f)\vert=\vert\pi_{0}(\ker f^{*})\vert=\vert\pi_{0}(K_2)\vert$.
\end{enumerate}
Note that we have fixed an injective character $\Qp/\Zp\hookrightarrow\Qlcl^\times$ throughout and hence we obtain the character $\psi:A\subseteq\Qp/\Zp\hookrightarrow\Qlcl^\times$. By construction, our original local system $\L$ equals $\L_{\psi}$ (see Remark \ref{decomposition}) and hence $H^*_c(U_1\times U_2,\L)$ can be identified as the $\psi$-isotypic component for the action of $A$ on $H^*_c(E,\Qlcl)$. 

Let us now describe the two above bases of $H^*_c(U_1\times U_2,\L(D))$ considered as the $\psi$-isotypic component of $H^{*}_{c}(E, \Qlcl(D))$ and use it to describe the relationship between the two bases. As the cohomology $H^*_c(U_1\times U_2,\L(D))$ is non-zero only in  degree $2D$, therefore it is enough to consider $H^{2D}_{c}(E, \Qlcl(D))$.
By Poincare duality (Remark \ref{Poincare Duality})
$$H^{2D}_{c}(E,\Qlcl(D))\cong
H^{2d}(E,\Qlcl(d))^{*} \mbox{ and  }H^{2D}_{c}(U_1\times U_2,\L(D))\cong
H^{2d}(U_1\times U_2,\L^\vee(d))^{*}.$$
Also by Remark \ref{Gysin Map}, there is a homomorphism $\CH^d(E)\rightarrow H^{2d}(E, \Qlcl (d))$.
We will describe the cohomology $H^{2d}(U_1\times U_2,\L^\vee(d))$ using the above cycle class map in terms of $\CH^d(E)$.

We have the following commutative diagram:
\begin{center}
\begin{tikzcd}
 E\vert_{\ker f\times U_{2}} \arrow[hookrightarrow]{r}\arrow[d]
& E \arrow[d, " \pi "]\arrow[hookleftarrow]{r}
&   E\vert_{U_{1}\times \ker f^{*}} \arrow[d]\\
\ker f\times U_{2} \arrow[hookrightarrow]{r} \arrow[u, bend left, "\sigma_{1}"]
& U_{1}\times U_{2}\arrow[hookleftarrow]{r} \
& U_{1}\times \ker f^{*}\arrow[l]\arrow[u, bend right, swap, "\sigma_{2}"]
\end{tikzcd}
\end{center}
where $\sigma_{1}\ \mbox{and}\ \sigma_{2}$ are trivializations and $\ker f\times U_{2}$, $U_{1}\times \ker f^{*}$ are defined by fibre product:
\begin{center}
\begin{tikzcd}
\ker f\times U_{2}\arrow[r]\arrow[d]
&{0}\times U_{2}\arrow[hookrightarrow]{d}\\
U_{1}\times U_{2}\arrow[r, "f\times Id"]
& U'_{2}\times U_{2}
\end{tikzcd}\ \ \ \ \ \ \ \ \ \ \ \ \ \
\begin{tikzcd}
U_{1}\times \ker f^{*}\arrow[r]\arrow[d]
&U_{1}\times {0}\arrow[hookrightarrow]{d}\\
U_{1}\times U_{2}\arrow[r, "Id\times f^{*}"]
& U_{1}\times U'_{1}
\end{tikzcd}
\end{center}
As before, let $ K_{1},\ K_{2}$ denote the reduced scheme associated with $\ker f$, $\ker f^{*}$ respectively. 
Let us fix some notation for further discussion. For any $b_{i}\in K_{i}$, let $ K_{i}^{b_{i}}$ denote the connected component of  $ K_{i}$ containing $b_{i}$.   For any $b_1\in K_1,b_2\in K_2$ and $a\in A$ define $X_{0}^{b_{2}}= \sigma_{2}(U_{1}\times K_{2}^{b_{2}}),\ Y_{0}^{b_{1}}=\sigma_{1}(K_{1}^{b_{1}}\times U_{2})$, $X_{a}^{b_{2}}= a\cdot X_{0}^{b_{2}},\ Y_{a}^{b_{1}}=a\cdot Y_{0}^{b_{1}}$. Hence we have the smooth subschemes $X^{b_2}_a,Y^{b_1}_a\subset E$ of codimension $d$.
We have an action of $G:=K_{1}\times K_{2}\times A$ on $E$, which induces an action of $G$ on Chow group.
Suppose $k_{1},b_{1}\in K_{1},\ k_{2},b_{2}\in K_{2}$ and $a,a_{1}\in A$ then we  see that
$$k_{1}\cdot X_{a}^{b_{2}}=X_{a+B(k_{1}, b_{2})}^{b_{2}} \ \ \ \ \ \ \ \ \ \  k_{1}\cdot Y_{a}^{b_{1}}=Y_{a}^{k_{1}+b_{1}}$$
 $$k_{2}\cdot X_{a}^{b_{2}}=X_{a}^{k_{2}+b_{2}}\ \ \ \ \  \ \ \ \ \ k_{2}\cdot Y_{a}^{b_{1}}=Y_{a-B(b_{1}, k_{2})}^{b_{1}}$$
 $$a_{1}\cdot X_{a}^{b_{2}}=X_{a+a_{1}}^{b_{2}}\ \ \ \ \ \ \ \ \ \ a_{1}\cdot Y_{a}^{b_{1}}=Y_{a+a_{1}}^{b_{2}}.$$
Define $X_{b_{2}}=\sum_{a\in A}\psi(a)X_{a}^{b_{2}}\  \mbox{and}\ Y_{b_{1}}=\sum_{a\in A}\psi(a)Y_{a}^{b_{1}}$ as elements of $C^d(E)\otimes \Qlcl$. 
\begin{remark}
By a slight abuse of notation, we will often use the same symbol to denote a $d$-cycle as well as its image in $H^{2d}(E,\Qlcl(d))$ under the cycle class map.
\end{remark}
By the description of the action of $G$ on these $d$-cycles we see that for $a_1\in A, b_i\in K_i$, we have $a_1\cdot X_{b_2}=\psi^{-1}(a_1)X_{b_2}$ and $a_1\cdot Y_{b_1}=\psi^{-1}(a_1)Y_{b_1}$ and that the two sets  $\mathcal{B}_{X}=\{X_{b_{2}}: b_{2}\in \pi_{0}(K_{2})\}$ and 
 $\mathcal{B}_{Y}=\{Y_{b_{1}}: b_{1}\in \pi_{0}(K_{1})\}$ form two bases of $H^{2d}(U_{1}\times U_{2},\L^{\vee}_{\psi}(d))$ considered as the $\psi^{-1}$-isotypic direct summand in $H^{2d}(E,\Qlcl(d))$.
By Remark \ref{Poincare Duality}, $H_{c}^{2D}(U_{1}\times U_{2},\L_{\psi}(D))$ is isomorphic to dual of $H^{2d}(U_{1}\times U_{2},\L_{\psi}^{\vee}(d))$.
 We will denote these dual bases of $H_{c}^{2D}(U_{1}\times U_{2},\L_{\psi}(D))$ by 
 $\mathcal{B}^{*}_{X}=\{X^{*}_{b_{2}}: b_{2}\in \pi_{0}(K_{2})\}$ and 
 $\mathcal{B}^{*}_{Y}=\{Y^{*}_{b_{1}}: b_{1}\in \pi_{0}(K_{1})\}$. These are precisely the two bases of $H_{c}^{2D}(U_{1}\times U_{2},\L_{\psi}(D))$ described earlier. To describe the relationship between these two bases, it will be sufficient to describe the relationship between the two dual bases $\mathcal{B}_X$ and $\mathcal{B}_Y$ of $H^{2d}(U_1\times U_2,\L^\vee_{\psi}(d))$.

\section {The special case $\Ga\times\Ga$}\label{additive group}
In this section we give an explicit description of the relationship between the two bases for biextensions of $\Ga\times\Ga$. Let $\Gap$ be the perfectization of the additive group $\Ga$.
In particular, $\Gap$ is the spectrum of the ring, 
$R = k[x, x^{1/p}, x^{1/{p^2}}, \dots ]$.
Let us consider the dual pair $(\Ga,\Ga,\E_{\Ga})$ with $\E_{\Ga}=\spec\left(k(x,y,z)/(xy-z^{p}+z)\right)$. This gives us an identification of the dual ${\Gap }^{*}$ with $\Gap$ (cf. \cite[1.2]{Datta}), i.e. using the above dual pair we can consider $\Ga$ as a model for $\Gap^*$. Now let $E$ be any biextension of $\Ga \times \Ga$ and $f : \Gap\rightarrow\Gap\cong \Gap^*$ the corresponding homomorphism. We have the following Lemma (cf. \cite[Lemma 1, Prop. 11]{Datta}):
\begin{lemma}
\begin{enumerate}
\item $\End(\Gap)= k\{\Phi, \Phi^{-1}\}$ where $\Phi$ is the Frobenius automorphism (which sends $x$ to $x^{p}$). 
The algebra structure is given by $\Phi a=a^{p}\Phi$ for all $a\in k$ and $\Phi^{*}=\Phi^{-1},\ c^{*}=c$ for all $c\in k$.
In particular, if  $f\in \End(\Gap)$, $f=\sum_{k=m}^{M}a_{k}\Phi^{k}$ then $f^{*}=\sum_{k=m}^{M}a^{p^{-k}}_{k}\Phi^{-k}=\sum_{k=-M}^{-m}{(a_{-k})}^{p^{k}}\Phi^{k}$. 
\item If $E$ is as above then $E\cong \mbox{Spec}\left(T[z]/(z^{p}-z-f(x)y)\right)$, where $T=R\otimes R$.
\item There exists a unique element $g(x,y)\in T$ such that $g(0,0)=0$ and $g(x,y)^{p}-g(x,y)=f(x)y-xf^{*}(y)$.  Furthermore, for $b_{1}\in\ker f$ and $b_{2}\in \ker f^{*}$, the pairing $B$ is given by the formula $B(a,b)=g(a,b)$.
\end{enumerate}
\end{lemma}

\begin{remark}
One has an explicit formula for the element $g$ above. If $f=a\tau^{n}$, then one has:
$$ g(x,y)=
\begin{cases}
a^{p^{-1}}x^{p^{n-1}}y^{p^{-1}}+\dots +a^{p^{-n}}xy^{p^{-n}}\ \ \ \ \ \ \mbox{if}\  n>0\\
0\ \ \ \ \ \ \ \ \ \ \ \ \ \ \ \ \ \ \ \  \ \ \ \ \ \ \ \ \ \ \ \ \ \ \ \ \ \ \ \ \ \  \ \ \ \ \ \  \mbox{if}\ n=0 \\
-ax^{p^{n}}y-\dots - a^{p^{-n-1}}x^{p^{-1}}y^{p^{-n-1}}\ \ \ \ \ \mbox{if}\ n<0 
\end{cases}$$
For general $f$, one notes that $g$ varies additively with respect to $f$.
\end{remark}

We want to give an explicit  description of the relations between two bases of cohomology in the case of bimultiplicative local systems on $\Ga\times\Ga$. It will be convenient to choose models for the duals $\Gap^*$ (which may differ from the standard model mentioned above) such that the corresponding homomorphisms $f,f^*:\Gap\to\Gap^*$ are algebraic maps with respect to these models.

Let $\L$ be a bimultiplicative $\Qlcl$-local system on $\Gap\times\Gap$ such that the corresponding homomorphisms $f,\ f^{*}:\Gap\to\Gap^*\cong\Gap$ are given by $f=\sum_{k=m}^{M}a_{k}\Phi^{k}$ and $f^{*}=\sum_{k=m}^{M}a^{p^{-k}}_{k}\Phi^{-k}$ and are non-zero.
Now we choose a different model $\Ga'=\spec(k[x^{p^{m}}])$ for $\Gap^{*}$ then we have an algebraic map $\tilde{f}:\Ga\rightarrow\Ga'$, $\tilde{f}(x)=\sum_{k=0}^{M-m}a_{k+m}^{p^{-m}}x^{p^{k}}$ such that $\tilde{f}_{\perf}=f$ and one more model $\Ga''=\spec(k[x^{p^{-M}}])$  for $\Gap^{*}$ then we have an algebraic map $\tilde{f^{*}}(x)=\sum_{k=0}^{M-m}a^{p^{-(k+m)}}_{M-k}x^{p^{-k}}$ such that $\tilde{f^{*}}_{\perf}=f^{*}$.
Therefore we get two dual pairs $(\Ga,\Ga',\mathcal{E}_{\Ga'})$ and $(\Ga,\Ga'',\mathcal{E}_{\Ga''})$.

We have the following commutative diagrams:
\begin{center}
\begin{tikzcd}
E_{\perf}  \arrow[d]
\\ \Gap\times\Gap
\end{tikzcd}
\begin{tikzcd}
E \arrow[r, "F"] \arrow[d]
& \tau^{*}(\mathcal{E}_{\mathbb{G}'_{a}})\arrow[d]\\
\mathbb{G}_{a}\times\mathbb{G}_{a}\arrow[r, "\tilde{f}\times Id"]
& 
\mathbb{G}'_{a}\times\mathbb{G}_{a}
\end{tikzcd}\ \ \ \ 
\begin{tikzcd}
E \arrow[r,"F'"] \arrow[d]
& \mathcal{E}_{\mathbb{G}''_{a}}\arrow[d, ]\\
\mathbb{G}_{a}\times\mathbb{G}_{a}\arrow[r, "Id\times \tilde{f^{*}}"]
& 
\mathbb{G}_{a}\times\mathbb{G}_{a}''
\end{tikzcd}
\end{center}
where $E_{\perf}\cong \spec\left(T[z]/(z^{p}-z-f(x)y)\right)$, $E_{\perf}\cong \spec\left(T[z]/(z^{p}-z-xf^{*}(y))\right)$, $\tau^{*}(\mathcal{E}_{\Ga'})=\{(x,y,z)\in \Ga'\times\Ga\times\Ga: x^{p^{m}}y=z^{p}-z\}$ and $\mathcal{E}_{\Ga''}=\{(x,y,z)\in \Ga\times\Ga''\times\Ga: xy^{p^{-M}}=z^{p}-z\}$.

Let $X'_{0}=\{(x,y,z)\in\tau^{*}(\mathcal{E}_{\Ga' }): y=z=0\}$ and $Y'_{0}=\{(x,y,z)\in\tau^{*}(\mathcal{E}_{\Ga'}): x=z=0\}$  be two 1-cycles in $\tau^{*}(\mathcal{E}_{\Ga' })$.
We have $ X'_{0}+p^{m}\cdot Y'_{0}=\mbox{div}(z)$, 
therefore in Chow group of $\tau^{*}(\mathcal{E}_{\Ga' })$, $ X'_{0}+p^{m}\cdot Y'_{0}=0$.
Similarly, for any $a\in \mathbb{F}_p\subset k$, set $X'_{a}=a\cdot X'_{0}=\{(x,y,z)\in\tau^{*}(\mathcal{E}_{\Ga' }): y=0,z=a\}$ and $Y'_{a}=a\cdot Y'_{0}=\{(x,y,z)\in\tau^{*}(\mathcal{E}_{\Ga' }): x=0,z=a\}$ then $\mbox{div}(z-a)=X'_{a}+p^{m}\cdot Y'_{a}$, so in Chow group $X'_{a}+p^{m}\cdot Y'_{a}=0$ for each $a\in \mathbb{F}_p$. By  applying  Lemma \ref{Pullback} to $F:E\rightarrow\tau^{*}(\E_{\Ga'})$, we get for each $a\in \mathbb{F}_p=A$
$$X_{a}^{0}+p^{m}\sum_{b_{1}\in \ker \tilde{f}}Y_{a}^{b_{1}}=0.$$
Similarly using $\mathcal{E}_{\Ga ''}$ and $F':E\rightarrow\E_{\Ga'}$ we get that for each $a\in \mathbb{F}_p=A$,
$$p^{-M}\sum_{b_{2}\in \ker\tilde{f^{*}}}X_{a}^{b_{2}}+Y_{a}^{0}=0
$$
Using these equation we get,
\begin{equation}\label{2}
0=\sum_{a\in A}\psi(a)\left( p^{-M}\sum_{b_{2}\in \ker \tilde{f^{*}}}X_{a}^{b_{2}}+Y_{a}^{0}\right)=p^{-M}\sum_{b_{2}\in\ker \tilde{f^{*}}}X_{b_{2}}+Y_{0}
\end{equation}
and 
\begin{equation}\label{3}
0=\sum_{a\in A}\psi(a)\left( p^{m}\sum_{b_{1}\in \ker \tilde{f}}Y_{a}^{b_{1}}+X_{a}^{0}\right)=p^{m}\sum_{b_{1}\in\ker \tilde{f}}Y_{b_{1}}+X_{0}.
\end{equation}
\begin{lemma}
Let $\mathcal{L}$ be a bimultiplicative $\Qlcl$-local system on $\Gap\times\Gap$ corresponding to   homomorphism $f,\ f^{*}$ as above then the relation between the two bases of $H^{2}(\Gap\times \Gap, \mathcal{L}^{\vee}(1))$ is
$$Y_{b_{1}}=-p^{-M}\sum_{b_{2}\in \pi_{0}(\ker \tilde{f}^{*})}\psi(-B(b_{1}, b_{2}))X_{b_{2}}$$
$$X_{b_{2}}=-p^{m}\sum_{b_{1}\in \pi_{0}(\ker \tilde{f})}\psi(B(b_{1}, b_{2}))Y_{b_{1}}$$
In other words, the change of basis matrix from ${\mathcal{B}_{X}}$ to ${\mathcal{B}_{Y}}$ is:
$$-p^{m}\cdot\begin{pmatrix} \psi(B(b_{1},b_{2}))
\end{pmatrix}^{T}_{b_{1}\in \pi_{0}(\ker f),\\ b_{2}\in \pi_{0}(\ker f^{*})}  \mbox{ and its inverse is }$$
$$-p^{-M}\cdot\begin{pmatrix} \psi(-B(b_{1},b_{2}))
\end{pmatrix}_{b_{1}\in \pi_{0}(\ker f),\\ b_{2}\in \pi_{0}(\ker f^{*})} $$
where $\psi: \mathbb{Z}/p\mathbb{Z}\rightarrow \Qlcl^{\times}$ is a fixed character, ${\mathcal{B}_{X}}=\{{X_{b_{2}}} : b_{2}\in \pi _{0}(\ker f^{*})\}$,  ${\mathcal{B}_{Y}}=\{{Y_{b_{1}}} : b_{1}\in \pi _{0}(\ker f)\}$ denote the two bases of $H^{2}(\Gap\times \Gap, \mathcal{L}^{\vee}(1))$ as explained in Section \ref{computation}.
\end{lemma}
\begin{proof}
From equation (\ref{2}),\ (\ref{3}), we have following relations in $H^{2}(\Gap\times \Gap, \mathcal{L}^{\vee}(1))$,
\begin{equation}\label{7}
p^{-M}\sum_{b_{2}\in\ker \tilde{f^{*}}}X_{b_{2}}+Y_{0}=0
\end{equation}and
\begin{equation}\label{8}
p^{m}\sum_{b_{1}\in\ker \tilde{f}}Y_{b_{1}}+X_{0}=0.
\end{equation}
Let $b_{1}\in \ker \tilde{f}$, $b_{2}\in \ker \tilde{f}^{*}$. By using action of $G\ (=\ker \tilde{f}\times\ker \tilde{f}^{*}\times\mathbb{Z}/p\mathbb{Z})$ on the Chow group of $E$, we have
$$b_{1}\cdot Y_{0}=Y_{b_{1}},\ \ \ b_{1}\cdot X_{b_{2}}=\psi(-B(b_{1},b_{2}))X_{b_{2}},\ \ \ b_{2}\cdot X_{0}=X_{b_{2}},\ \ \ b_{2}\cdot Y_{b_{1}}=\psi(B(b_{1},b_{2}))Y_{b_{1}}.$$
Using above relations and equation (\ref{7}) and (\ref{8}) we get that,
$$p^{-M}\sum_{b_{2}\in\ker \tilde{f^{*}}}\psi(-B(b_{1},b_{2}))X_{b_{2}}+Y_{b_{1}}=0$$
$$p^{m}\sum_{b_{1}\in\ker \tilde{f}}\psi(B(b_{1},b_{2})) Y_{b_{1}}+X_{b_{2}}=0$$
This proves the lemma.
\end{proof}

\section{The case of universal biextensions and dual pairs}
Let $U$ be a connected commutative unipotent algebraic group over $k$.
Let $\E_{U_{\perf}}$ be the universal biextension of $U_{\perf}\times U_{\perf}^{*}$ corresponding to $Id:U_{\perf}\rightarrow U_{\perf}$.
We can choose a model $U'$ for $U_{\perf}^{*}$ and hence a dual pair $(U,U',\E_{U})$ whose perfectization corresponds to the universal biextension $\E_{U_{\perf}}$ of $U_{\perf}\times U_{\perf}^{*}$.
Let $\L$ be the bimultiplicative $\Qlcl$-local system on $U\times U'$ such that $\L_{\perf}$ is a bimultiplicative $\Qlcl$-local system $U_{\perf}\times U_{\perf}^{*}$  corresponding to the universal biextension of $U_{\perf}\times U_{\perf}^{*}$.
Therefore by \ref{dim eq},  dimension of $H_{c}^{2\dim(U)}(U\times U', \mathcal{L}(\dim(U)))$ is $1$ and $H_{c}^{2\dim(U)}(U\times U', \mathcal{L}(\dim(U)))=\langle X_{0}^{*}\rangle=\langle {Y_{0}}^{*}\rangle$, 
where ${X_{0}}^{*}$ and ${Y_{0}}^{*}$ denotes the dual basis elements of $H_{c}^{2\dim(U)}(U_{\perf}\times {U_{\perf}}^{*}, \mathcal{L}(\dim(U)))$ corresponding to $X_{0}$ and  $Y_{0}$ respectively as in Section \ref{computation}.

\begin{definition}\label{multiplicity:universal}
For each dual pair $(U,U',\E_{U})$ we will associate a number $m(U,U',\E_{U})$ such that $m(U,U',\E _{U})$ satisfies  the  relation $X^{*}_{0}=\frac{1}{m(U,U',\E_{U})}\cdot Y^{*}_{0}$ in  $H_{c}^{2\dim(U)}(U\times U', \mathcal{L}(\dim(U)))$, or equivalently, in $H^{2\dim(U)}(U\times U',\L^{\vee}(\dim(U)))$, we have $X_{0}=m(U,U',\E_{U})\cdot Y_{0}$, where $\mathcal{L}$ is the bimultiplicative $\Qlcl$-local system on $U\times U'$ corresponding to the biextension $\E_{U}$ of $U\times U'$.
\end{definition}
\begin{remark}
Let $(U,U',\mathcal{E}_{U})$ be a dual pair and $\tau:U'\times U\rightarrow U\times U'$ such that $\tau(x,y)=(y,x)$.
 We will write $\tau^{*}(\mathcal{E}_{U})$ for  the fiber product of  $\tau$ and the morphism $\mathcal{E}_{U}\rightarrow U\times U'$.
 Then $(U',U,\tau^{*}(\mathcal{E}_{U}))$ is a dual pair and $m(U',U,\tau^{*}(\mathcal{E}_{U}))=\frac{1}{m(U,U',\E _{U})}$.
\end{remark}
\begin{remark}\label{multiplicity: Ga}
Let $(\Ga',\Ga'', \E)$  be a dual pair such that we have the following cartesian  square:
\begin{center}
\begin{tikzcd}
\E \arrow[r] \arrow[d]
& \Ga \arrow[d, "L"]\\
\Ga'\times\Ga''\arrow[r, "m"]
& 
\Ga
\end{tikzcd}
\end{center}
where $\Ga'=\spec(k[x^{p^{r}}]),\ \Ga''=\spec(k[y^{p^{s}}]),\ m(x,y)=x^{p^{r}}y^{p^{s}}$ for some $r,s \in \mathbb{Z}$ and $L$ is the Lang isogeny which sends $z$ to $z^{p}-z$ and $\E=\{(x,y,z)\in \Ga'\times\Ga''\times\Ga:x^{p^{r}}y^{p^{s}}=z^{p}-z \}$. 
In the cohomology $H^{2}_{c}(\Ga'\times\Ga'',\L(1))$, we have  $X^{*}_{0}=-p^{-(r-s)}\cdot Y^{*}_{0}$ \textit{i.e.} $m(\Ga',\Ga'', \E)=-p^{r-s}$, where $\L$ is a bimultiplicative $\Qlcl$-local system on $\Ga'\times\Ga''$ corresponding to biextension $\E$.
\end{remark}

Let $0\rightarrow \tilde{V}\rightarrow \tilde{U}\rightarrow \tilde{W}\rightarrow 0$ be a short exact sequence of perfect connected commutative unipotent groups. By duality (\ref{dualses}), we get another short exact sequence $0\rightarrow \tilde{W}^{*}\rightarrow \tilde{U}^{*}\rightarrow \tilde{V}^{*}\rightarrow 0$.
Let fix a dual pair $(U,U',\mathcal{E}_{U})$ $i.e.$ we fix a model for $\tilde{U} \ \mbox{and}\ \tilde{U}^{*}$ and in this section we consider $\E_{U}$ as biextension of $U\times U'$ by a finite group $A\leq\Qp/\Zp$.
 By Remark \ref{model:subgroup}, we can choose the models for $\tilde{V}$ and $\tilde{W}^{*}$, say $V$ and $W'$ such that $0\rightarrow V\rightarrow U$ and $0\rightarrow W'\rightarrow U'$ are exact.
Then $W=U/V$ and $V'=U'/W'$ give the models for $\tilde{W}$ and $\tilde{V}^{*}$ respectively.
Therefore we get the dual pairs 
$(V,V',\mathcal{E}_{V})\  \mbox{and}\ (W, W',\mathcal{E}_{W})$.

\begin{theorem}\label{SES}
 Let $0 \rightarrow V_{\perf} \xrightarrow{f} U_{\perf} \xrightarrow{g} W_{\perf} \rightarrow 0$ be a short exact sequence of perfect connected commutative   unipotent group scheme over $k$.  Let $(U,U',\mathcal{E}_{U}),\ (V,V',\mathcal{E}_{V})\ \mbox{and}\ (W, W',\mathcal{E}_{W})$ be the dual pairs as above.
  Then  $m(U,U',\mathcal{E}_{U})=m(V,V',\mathcal{E}_{V})\cdot 
  m(W, W',\mathcal{E}_{W})$.
\end{theorem}
\begin{proof}
We have $0 \rightarrow V_{\perf} \xrightarrow{f} U_{\perf} \xrightarrow{g} W_{\perf} \rightarrow 0$, by duality (\ref{dualses})  $0 \rightarrow W'_{\perf} \xrightarrow{g^{*}} U'_{\perf} \xrightarrow{f^{*}} V'_{\perf} \rightarrow 0$.  
Using these two short exact sequences and discussion above, we get short exact sequences of commutative connected unipotent algebraic group scheme over $k$,
$0 \rightarrow V\xrightarrow{\tilde{f}} U \xrightarrow{\tilde{g}} W \rightarrow 0$  and $0 \rightarrow W'\xrightarrow{\tilde{g^{*}}} U' \xrightarrow{\tilde{f^{*}}} V' \rightarrow 0$. 
Note that  all models are  reduced group schemes by definition.
We have the following commutative diagram of cartesian squares,
\begin{center}
\begin{tikzcd}
\mathcal{E}_{V}\arrow[d]
& \mathcal{E}_{U}\vert_{V\times U'} \arrow[d]\arrow[l, two heads, swap,"F"]\arrow[hookrightarrow,"j"]{r}
&   \mathcal{E}_{U}\arrow[d]\arrow[hookleftarrow,"j_{1}"]{r}
& \mathcal{E}_{U}\vert_{U\times W'} \arrow[d]\arrow[r, two heads,"F_{1}"]
&\mathcal{E}_{W}\arrow[d]\\
V\times V'
& V\times U'\arrow[l, two heads, swap, "Id\times \tilde{f^{*}}"]\arrow[hookrightarrow, "\tilde{f}\times Id" ]{r}
&U\times U'\arrow[hookleftarrow, "Id\times \tilde{g^{*}}" ]{r}
&U\times W'\arrow[r, two heads, "\tilde{g}\times Id"]
& W\times W'
\end{tikzcd}
\end{center}
where all vertical maps are $A$-torsors for a finite subgroup $A\leq\Qp/\Zp$.
Consider the following commutative diagrams:
\begin{center}
\begin{tikzcd}
 \E_{W}\vert_{0\times W'} \arrow[hookrightarrow]{r}\arrow[d]
& \E_{W} \arrow[d]\arrow[hookleftarrow]{r}
&   \E_{W}\vert_{W\times 0} \arrow[d]\\
0\times W' \arrow[hookrightarrow]{r} \arrow[u, bend left, "\sigma^{W}_{1}"]
& W\times W'\arrow[hookleftarrow]{r} \
& W\times 0\arrow[l]\arrow[u, bend right, swap, "\sigma^{W}_{2}"]
\end{tikzcd}
\begin{tikzcd}
 \E_{U}\vert_{\tilde{f}(V)\times W'} \arrow[hookrightarrow]{r}\arrow[d]
& \E_{U}\vert_{U\times W'}\arrow[d]\arrow[hookleftarrow]{r}
&  \E_{U}\vert_{U\times 0} \arrow[d]\\
\tilde{f}(V')\times W' \arrow[hookrightarrow]{r} \arrow[u, bend left, "\sigma^{VW'}_{1}"]
& U\times W'\arrow[hookleftarrow]{r} \
& U\times 0\arrow[l]\arrow[u, bend right, swap, "\sigma^{UW'}_{2}"]
\end{tikzcd}
\end{center}
where $\sigma_{1}^{W}$, $\sigma_{1}^{W}$,$\sigma^{UW'}_{1}$,$\sigma^{UW'}_{2}$ are trivializations. 
We will use same notations as defined in Section \ref{computation}, for $a\in A$ $X^{0,W}_{a}=a\cdot\sigma_{2}^{W}(W\times 0)$, $Y^{0,W}_{a}=a\cdot\sigma_{1}^{W}(0\times W')$, 
$X_{0}^{W}=\sum_{a\in A}\psi(a)X^{0,W}_{a}$, $Y_{0}^{W}=\sum_{a\in A}\psi(a)Y^{0,W}_{a}$ and 
we will denote $a\cdot\sigma_{1}^{UW'}( \tilde{f}(V')\times W')$ by $[\tilde{f}(V')\times W'\times {a}]$.

Let $\L_{W}$ be a bimultiplicative $\Qlcl$-local system on $W\times W$ corresponding to biextension $\E_{W}$ of $W\times W$.
Consider $H^{2\dim(W)}(W\times W',\L^{\vee}_{W}(\dim W))$ as a subspace of $H^{2\dim(W)}(\E_{W},\Qlcl(\dim(W)))$ then we have
$H^{2\dim(W)}(W\times W',\L^{\vee}_{W}(\dim W))=\langle X^{W}_{0}\rangle=\langle Y^{W}_{0}\rangle$ and
$X_{0}^{W}=m(W,W',\E_{W})\cdot Y_{0}^{W}$.
By Lemma \ref{Pullback}, $F^{*}_{1}(X^{0,W}_{0})=\sigma_{2}^{UW'}(U\times 0)$ and $F^{*}_{1}(Y^{0,W}_{0})=\sigma_{1}^{UW'}( \tilde{f}(V')\times W')=[\tilde{f}(V')\times W'\times {0}]$.
By applying Remark \ref{9.2} to $F_{1}$ and Remark \ref{9.3} to $j_{1}$, we get
\begin{equation}\label{eq:W}
X^{U}_{0}=m(W,W',\E_{W})\sum_{a\in A}\psi(a)[\tilde{f}(V')\times \tilde{g^{*}}(W')\times {a}]
\end{equation}

Consider the another commutative diagrams:
\begin{center}
\begin{tikzcd}
 \E_{V}\vert_{0\times V'} \arrow[hookrightarrow]{r}\arrow[d]
& \E_{V} \arrow[d]\arrow[hookleftarrow]{r}
&   \E_{V}\vert_{V\times 0} \arrow[d]\\
0\times V' \arrow[hookrightarrow]{r} \arrow[u, bend left, "\sigma^{V}_{1}"]
& V\times V'\arrow[hookleftarrow]{r} \
& V\times 0\arrow[l]\arrow[u, bend right, swap, "\sigma^{V}_{2}"]
\end{tikzcd}
\begin{tikzcd}
 \E_{V}\vert_{0\times U'} \arrow[hookrightarrow]{r}\arrow[d]
& \E_{V}\vert_{V\times U'}\arrow[d]\arrow[hookleftarrow]{r}
&  \E_{V}\vert_{V\times \tilde{g^{*}}(W')} \arrow[d]\\
0\times U' \arrow[hookrightarrow]{r} \arrow[u, bend left, "\sigma^{VU'}_{1}"]
& V\times U'\arrow[hookleftarrow]{r} \
& V\times \tilde{g^{*}}(W')\arrow[l]\arrow[u, bend right, swap, "\sigma^{VU'}_{2}"]
\end{tikzcd}
\end{center}
where $\sigma_{1}^{V}$, $\sigma_{1}^{V}$,$\sigma^{VU'}_{1}$ $\sigma^{VU'}_{2}$ are trivializations. 
We will use same notations as defined in Section \ref{computation}, for $a\in A$ $X^{0,V}_{a}=a\cdot\sigma_{2}^{V}(V\times 0)$, $Y^{0,V}_{a}=a\cdot\sigma_{1}^{V}(0\times V')$, 
$X_{0}^{V}=\sum_{a\in A}\psi(a)X^{0,V}_{a}$, $Y_{0}^{V}=\sum_{a\in A}\psi(a)Y^{0,V}_{a}$
and we will denote $a\cdot\sigma_{1}^{VU'}( V\times\tilde{g^{*}}(W')$ by $[V\times \tilde{g^{*}}(W')\times {a}]$.

Let $\L_{V}$ be a bimultiplicative $\Qlcl$-local system on $V\times V'$ corresponding to the biextension $\E_{V}$ of $V\times V'$.
Consider $H^{2\dim(V)}(V\times V',\L^{\vee}_{V}(\dim V))$ as a subspace of $H^{2\dim(V)}(\E_{V},\Qlcl(\dim(V)))$ then we have $H^{2\dim(V)}(V\times V',\L^{\vee}_{V}(\dim V))=\langle X^{V}_{0}\rangle=\langle Y^{V}_{0}\rangle$ and $X_{0}^{V}=m(V,V',\E_{V})\cdot Y_{0}^{V}$.
By Lemma \ref{Pullback}, $F^{*}(X^{0,V}_{0})=\sigma_{2}^{VU'}(V\times \tilde{g^{*}}(W'))=[V\times\tilde{g^{*}}(W')\times {0}]$ and $F^{*}(Y^{0,V}_{0})=\sigma_{1}^{VU'}(0\times U')$.
By applying Remark \ref{9.2} to $F$ and Remark \ref{9.3} to $j$, we get

\begin{equation}\label{eq:V}
\sum_{a\in A}\psi(a)[\tilde{f}(V)\times \tilde{g^{*}}(W')\times {a}]=m(V,V',\E_{V})Y^{U}_{0}
\end{equation}
Comparing equations (\ref{eq:W}) and (\ref{eq:V}) we get, $X^{U}_{0}=m(V,V',\E_{V})m(W,W',\E_{W})\cdot Y^{U}_{0}$ in $H^{2\dim U}(\E_{U},\Qlcl(\dim U))$
and $H^{2\dim U}(U\times U',\L^{\vee}_{U}(\dim U))=\langle X^{U}_{0}\rangle=\langle Y^{U}_{0}\rangle
$.
This proves  the theorem.

\end{proof}

\begin{corollary}\label{universal}
Let $U$ be a  connected commutative unipotent algebraic group over $k$ of dimension $n$ and let $(U,U',\mathcal{E}_{U})$ be a dual pair.   Then $m(U,U',\E_{U})=(-1)^{n}p^{r}$, for some $r \in \mathbb{Z}$.
\end{corollary}
\begin{proof}
 Corollary follows from the  previous Theorem \ref{SES}, Remark \ref{multiplicity: Ga}, an induction on dimension of $U$ and a well known fact about the unipotent algebraic group,  every connected commutative algebraic unipotent group scheme over $k$  has a subgroup isomorphic to $\Ga$.
\end{proof}

\section{Proof of main results}\label{sec:proof}
In this section we complete the proofs of the main results of this paper.
\begin{remark}{(cf.\cite[1.5]{Intersection})}\label{multiplicity} Let $X$ be a scheme of finite type over $k$ and $X_{1},\cdots,X_{r}$ be the irreducible components of $X$. 
The local rings $\mathcal{O}_{X_{i},X}$ all are zero dimensional (Artinian).
The \textit{geometric multiplicity} $m_{i}$ of $X_{i}$ in $X$ is defined to be the length of $\mathcal{O}_{X_{i},X}$ and the cycle $[X]=\sum_{i=1}^{r}m_{i}[X_{i,red}]$. 

Let $f:U_{1}\rightarrow U_{2}$ be an isogeny of connected unipotent algebraic groups, then we define  $\nu(f,U_{1},U_{2})$ to be the length of local ring $\mathcal{O}_{0,\ker(f)}$, therefore $[\ker(f)]=\sum_{b\in\ker(f)}\nu(f,U_{1},U_{2})[b_{red}]$. 
\end{remark}
\begin{remark}\label{def:constant}
Let $U_{1},\ U_{2}$ be two connected commutative unipotent algebraic groups. Let $f:U_{1,\perf}\rightarrow U_{2,\perf}$ be any isogeny between their perfectizations. By Remark \ref{model}, $f=\Phi^{-N}_{U_{2}/k}\circ f'_{\perf}$ for some isogeny $f':U_{1}\rightarrow U_{2}^{(p^{N})}$.
Then we define $\nu(f,U_{1},U_{2}):=p^{-N\cdot\dim(U_{2})}\cdot\nu(f',U_{1},U_{2}^{(p^{N})})$.  Note that this is an integral power of $p$.
\end{remark}
Now let $(U_{1},U'_{1},\E_{U_{1}})$ and $(U_{2},U'_{2},\E_{U_{2}})$ be two dual pairs giving us identifications ${U_{i,\perf}}^*\cong U'_{i,\perf}$.  Let $\L$ be a bimultiplicative $\Qlcl$-local system on  $U_{1}\times U_{2}$ such that the corresponding homomorphisms $f:U_{1,\perf}\rightarrow U'_{2,\perf},\ f^{*}:U_{2,\perf}\rightarrow U'_{1,\perf}$ are isogenies. We define the constants $m(U_{1},U_{2},\L)$ and $m(U_{2},U_{1},\tau^{*}(\L))$ as follows:
\begin{equation}
m(U_{1},U_{2},\L)=\frac{m(U_{2},U'_{2},\E_{U_{2}})}{\nu(f,U_{1},U'_{2})}\ \mbox{and}\ m(U_{2},U_{1},\tau^{*}(\L))=\frac{m(U_{1},U'_{1},\E_{U_{1}})}{\nu(f^{*},U_{2},U'_{1})}.
\end{equation}
By Corollary \ref{universal}, it follows that the number 
\begin{equation}\label{eq:const}
m(U_{1},U_{2},\L)={(-1)^{\dim(U_{2})}}\cdot{p^{r}}, \mbox{ for some }r\in \mathbb{Z}.
\end{equation}

\begin{remark}
The constants $m(U_{1},U_{2},\L)$ and $m(U_{2},U_{1},\tau^{*}(\L))$ in fact do not depend on the choices of the dual pairs made above. 
This is because, if we choose different dual pairs, then the numerators and denominators both scale by the same factors. This fact also follows from the next Lemma.
\end{remark}

\begin{lemma}\label{isogeny}
Let $U_{1}$ and $U_{2}$ be two  connected commutative unipotent algebraic groups over $k$ of dimension $d$.
Let $\L$ be a bimultiplicative $\Qlcl$-local system on $U_{1}\times U_{2}$ such that the corresponding homomorphism $f:U_{1,\perf}\rightarrow {U_{2,\perf}}^{*}$ is an isogeny. Then the relations between two bases of $H_{c}^{2d}(U_{1}\times U_{2}, \mathcal{L}(d))$ is given by,
$${Y^{*}_{b_{1}}}=\frac{1}{m(U_{1},U_{2},\L)}\sum_{b_{2}\in \pi_{0}(\ker f^{*})}\psi(B(b_{1},b_{2})){X^{*}_{b_{2}}}$$ and
$${X^{*}_{b_{2}}}=\frac{1}{m(U_{2},U_{1},\tau^{*}(\L))}\sum_{b_{1}\in \pi_{0}(\ker f)}\psi(-B(b_{1},b_{2})){Y^{*}_{b_{1}}},$$
where  $\psi$ is the fixed character $\Qp/\Zp\hookrightarrow \Qlcl^\times$. Also $m(U_{2},U_{1},\tau^{*}(\L))\cdot m(U_{1},U_{2},\L)=|\ker(f)|=|\ker(f^{*})|$.
\end{lemma}
\begin{proof}
Let us choose models $U'_{1}$ for ${U_{1,\perf}}^{*}$ and $U'_{2}$ for ${U_{2,\perf}}^*$ such that we have algebraic morphisms $f':U_{1}\rightarrow U'_{2}$, $f'^{*}:U_{2}\rightarrow U'_{1}$ whose perfectizations correspond to the given bimultiplicative local system $\L$ and the models of duals give rise to dual pairs  $(U_{1}, U'_{1},\mathcal{E}_{U_{1}})$, $(U_{2}, U'_{2},\mathcal{E}_{U_{2}})$. 
By Remark \ref{rk:bimultbiext}, $\L$  corresponds to a biextension of $U_{1}\times U_{2}$ by $\Qp/\Zp$. 
Let $\pi:E\rightarrow U_{1}\times U_{2}$ denotes the connected component of above biextension, which is a biextension of $U_{1}\times U_{2}$ by a finite subgroup $A$ of $\Qp/\Zp$ and $\pi^{*}(\L)$ is a trivial local system  and $\L$ equals to $\L_{\psi}$, so $\L^{\vee}$ equals $\L_{\psi^{\vee}}$ for  our fixed character $\psi:A\hookrightarrow \Qlcl^\times$.
 
 Consider the following  commutative diagram of pullback squares,
\begin{center}
\begin{tikzcd}
\mathcal{E}_{U_{1}}\arrow[d]
& E\arrow[d]\arrow[r]\arrow[l]
&\tau^{*}(\mathcal{E}_{U_{2}})\arrow[d]\\
U_{1}\times U_{1}'
&U_{1}\times U_{2}\arrow[l, swap, "Id\times f'^{*}"]\arrow[r, "f'\times Id"]
&U_{2}'\times U_{2}.
\end{tikzcd}
\end{center}

By Corollary \ref{universal} and Lemma \ref{Pullback} we get the following relations in $H^{2d}(E,\Qlcl(d))$:
$$\nu(f^{*},U_{2}, U'_{1})\sum_{b_{2}\in \pi_{0}(\ker f'^{*})}X^{b_{2}}_{0}-{m(U_{1},U'_{1},\E_{U_{1}})}Y^{0}_{0}=0$$ and
$$m(U'_{2},U_{2},\tau^{*}(\E_{U_{2}}))\cdot\nu(f,U_{1}, U_{2})\sum_{b_{1}\in \pi_{0}(\ker f')}Y^{b_{1}}_{0}-X^{0}_{0}=0.$$
That is,
$$Y^{0}_{0}=\frac{\nu(f^{*},U_{2}, U'_{1})}{m(U_{1},U'_{1},\E_{U_{1}})}\sum_{b_{2}\in \pi_{0}(\ker f'^{*})}X^{b_{2}}_{0}$$ and
$$X^{0}_{0}=\frac{\nu(f,U_{1},
U'_{2})}{m(U_{2},U'_{2},\E_{U_{2}})}\sum_{b_{1}\in \pi_{0}(\ker f')}Y^{b_{1}}_{0}.$$
By using the action of $A$ we get that,
$$Y^{0}_{a}=\frac{\nu(f^{*},U_{2},U'_{1})}{m(U_{1},U'_{1},\E_{U_{1}})}\sum_{b_{2}\in \pi_{0}(\ker f'^{*})}X^{b_{2}}_{a}$$ and
$$X^{0}_{a}=\frac{\nu(f,U_{1},U'_{2})}{m(U_{2},U'_{2},\E_{U_{2}})}\sum_{b_{1}\in \pi_{0}(\ker f')}Y^{b_{1}}_{a}.$$
Therefore, 
$$Y_{0}=\frac{\nu(f^{*},U_{2}, U'_{1})}{m(U_{1},U'_{1},\E_{U_{1}})}\sum_{b_{2}\in \pi_{0}(\ker f'^{*})}X_{b_{2}},$$
$$X_{0}=\frac{\nu(f,U_{1}, U'_{2})}{m(U_{2},U'_{2},\E_{U_{2}})}\sum_{b_{1}\in \pi_{0}(\ker f')}Y_{b_{1}}.$$
By using an action of $G$ and $\pi_{0}(\ker f)=\pi_{0}(\ker f')$, $\pi_{0}(\ker f^{*})=\pi_{0}(\ker f'^{*})$ we get the following relations
$$Y_{b_{1}}=\frac{1}{m(U_{2},U_{1},\tau^{*}(\L))}\sum_{b_{2}\in \pi_{0}(\ker f^{*})}\psi(-B(b_{1},b_{2}))X_{b_{2}},$$
$$X_{b_{2}}=\frac{1}{m(U_{1},U_{2},\L)}\sum_{b_{1}\in \pi_{0}(\ker f)}\psi(B(b_{1},b_{2}))Y_{b_{1}}$$
between the two bases of  $H^{2d}(U_1\times U_2,\L_{\psi^{\vee}}(d))$. Taking the dual bases and using Poincare duality completes the proof of the theorem.
\end{proof}
Let us finally consider the general situation considered in the introduction. Let $U_{1}$ and $U_{2}$ be two connected commutative perfect unipotent algebraic group scheme over $k$.
Let $E$ be a biextension of $U_{1} \times U_{2}$ by a $\Qp/\Zp$ and $f:U_{1}\rightarrow U_{2}^{*}, \ f^{*}:U_{2}\rightarrow U_{1}^{*}$ be the corresponding homomorphisms, where we no longer assume them to be isogenies. 
Let $K_{1},\ K_{2}$ denote the reduced scheme associated with kernels of $f$ and $f^{*}$ respectively.
The morphisms $f$ and $f^{*}$ factor as in the commutative diagrams below:
\begin{center}
\begin{tikzcd}
&U_{1}\arrow[r, "f"]\arrow[d, "\pi_{1}", swap]
&Im(f)\subseteq U_{2}^{*}\\
&U_{1}/K_{1}^{\circ}\arrow[ur, "f' ", swap]
\end{tikzcd}
\begin{tikzcd}
&U_{2}\arrow[r, "f^{*}"]\arrow[d, "\pi_{2}", swap]
&Im(f^{*})\subseteq U_{1}^{*}\\
&U_{2}/K_{2}^{\circ}.\arrow[ur, swap, "f'^{*} "]
\end{tikzcd}
\end{center}
By applying Serre duality we get that,
\begin{center}
\begin{tikzcd}
&Im(f)^{*}\arrow[r, "f^{*}"]\arrow[d, swap, "f'^{*} "]
&U_{1}^{*}\\
&(U_{1}/K_{1}^{\circ})^{*}\arrow[ur]
\end{tikzcd}
\begin{tikzcd}
&Im(f^{*})^{*}\arrow[r, "f"]\arrow[d, swap,"f' "]
&U_{2}^{*}\\
&(U_{2}/K_{2}^{\circ})^{*}\arrow[ur]
\end{tikzcd}
\end{center}

From this one can easily deduce that, $(U_{2}/K_{2}^{\circ})\cong \mbox{Im}(f)^{*}$ and $(U_{1}/K_{1}^{\circ})\cong \mbox{Im}(f^{*})^{*}$ and $f',f'^{*}$ are isogenies.
Therefore the biextension $E$ of $U_{1}\times U_{2}$ by $\Qp/\Zp$  descends to a biextension $E'$ of $U_{1}/K_{1}^{\circ}\times U_{2}/K_{2}^{\circ}$ by $\Qp/\Zp$ through the canonical projection morphisms and the corresponding morphisms $f'$ and $f'^{*}$.
Therefore we have a following cartesian square:
\begin{center}
\begin{tikzcd}
&E\arrow[r]\arrow[d]
&E'\arrow[d]\\
&U_{1}\times U_{2}\arrow[r]
&U_{1}/K_{1}^{\circ}\times U_{2}/K_{2}^{\circ}
\end{tikzcd}
\end{center}
\begin{remark}\label{rk:constants}
Let $(U_{1},U'_{1},\E_{U_{1}})$ and $(U_{2},U'_{2},\E_{U_{2}})$ be two dual pairs giving us identifications ${U_{i,\perf}}^*\cong U'_{i,\perf}$.  
Let $\L$ be a bimultiplicative $\Qlcl$-local system on  $U_{1}\times U_{2}$ and the corresponding  homomorphisms are $f:U_{1,\perf}\rightarrow U'_{2,\perf},\ f^{*}:U_{2,\perf}\rightarrow U'_{1,\perf}$. We define the constants $m(U_{1},U_{2},\L)$ and $m(U_{2},U_{1},\tau^{*}(\L))$ as follows:
$$m(U_{1},U_{2},\L)= m(U_{1}/K_{1}^\circ,U_{2}/K_{2}^\circ,\L)\ \ \ \ \ m(U_{2},U_{1},\tau^{*}(\L))=m(U_{2}/K_{2}^\circ,U_{1}/K_{1}^\circ,\tau^{*}(\L)$$
where $K_{1}$ and $K_{2}$ denote the reduced schemes associated with $\ker f$ and $\ker f^{*}$ respectively. As before, let $d=\dim (U_1/K_1^\circ)=\dim(U_2/K_2^\circ)$. Note that by (\ref{eq:const}) and Lemma \ref{isogeny} it follows that 
$$m(U_{1},U_{2},\L)=(-1)^d\cdot p^r \mbox{ and } m(U_{2},U_{1},\tau^{*}(\L))=(-1)^d\cdot p^{r'} \mbox{ for some }r,r'\in\mathbb{Z}$$ 
with $p^{r+r'}=|\pi_{0}(\ker f)|=|\pi_{0}(\ker f^{*})|$. (See also Remark \ref{rk:constsrr'}.)
\end{remark}

\noindent Now we complete the proof of Theorem \ref{MAIN RESULT}.
\begin{proof}
[\bf{Proof of Theorem \ref{MAIN RESULT}}]\label{proof of 1.1}
In the above discussion if we replace $E$ by connected component of $E$ and $E'$ by its image then we have a following cartesian square:
\begin{center}
\begin{tikzcd}
&E\arrow[r]\arrow[d]
&E'\arrow[d]\\
&U_{1}\times U_{2}\arrow[r,"\pi_{1}\times\pi_{2}"]
&U_{1}/K_{1}^{\circ}\times U_{2}/K_{2}^{\circ}
\end{tikzcd}
\end{center}
Also, the biextension $E'$ of $U_{1}/K_{1}^{\circ}\times U_{2}/K_{2}^{\circ}$ by a finite commutative group $A\ (\leq \Qp/\Zp)$ such that $\L$ is trival on it and bimultiplicative $\Qlcl$-local system $\L$ corresponds to a character $\psi:A\hookrightarrow \Qlcl^{\times}$.
Then by Lemma \ref{isogeny}, we have following relations in $H^{2d}(U_{1}/K_{1}^{0}\times U_{2}/K_{2}^{0},\L_{\psi^{\vee}}(d))$:
$$Y'_{b_{1}}=\frac{1}{m(U_{2}/K_{2}^\circ,U_{1}/K_{1}^\circ,\tau^{*}(\L))}\sum_{b_{2}\in \pi_{0}(\ker f'^{*})}\psi(-B(b_{1},b_{2}))X'_{b_{2}},$$
$$X'_{b_{2}}=\frac{1}{m(U_{1}/K_{1}^\circ,U_{2}/K_{2}^\circ,\L)}\sum_{b_{1}\in \pi_{0}(\ker f')}\psi(B(b_{1},b_{2}))Y'_{b_{1}}.$$
Now by applying Lemma \ref{Pullback} we get the following relations in $H^{2d}(U_{1}\times U_{2},\L_{\psi^{\vee}}(d))$:
$$Y_{b_{1}}=\frac{1}{m(U_{2}/K_{2}^\circ,U_{1}/K_{1}^\circ,\tau^{*}(\L)}\sum_{b_{2}\in \pi_{0}(\ker f^{*})}\psi(-B(b_{1},b_{2}))X_{b_{2}},$$
$$X_{b_{2}}=\frac{1}{m(U_{1}/K_{1}^\circ,U_{2}/K_{2}^\circ,\L)}\sum_{b_{1}\in \pi_{0}(\ker f)}\psi(B(b_{1},b_{2}))Y_{b_{1}}.$$
\emph{i.e.},
$$Y_{b_{1}}=\frac{1}{m(U_{2},U_{1},\tau^{*}(\L)}\sum_{b_{2}\in \pi_{0}(\ker f^{*})}\psi(-B(b_{1},b_{2}))X_{b_{2}},$$
$$X_{b_{2}}=\frac{1}{m(U_{1},U_{2},\L)}\sum_{b_{1}\in \pi_{0}(\ker f)}\psi(B(b_{1},b_{2}))Y_{b_{1}}.$$
Taking the dual bases and using Poincare duality completes the proof of the theorem.
\end{proof}

\begin{proof}[\bf{Proof of Corollary \ref{nondegenerate}}]
Suppose, $b_{1}\in \pi_{0}(\ker f)$ such that $B(b_{1},b)=0\ \forall\ b\in \pi_{0}(\ker f^{*})$. By Theorem \ref{MAIN RESULT}, we have 
$Y_{b_{1}}=Y_{0}$
which implies that $b_{1}=0$.
Similarly one can show that if $b_{2}\in \pi_{0}(\ker f^{*})$ such that $B(b,b_{2})=0\ \forall\ b\in \pi_{0}(\ker f)$ then $b_{2}=0$.
This proves the corollary.
\end{proof}

\begin{corollary}\label{generic character}
Let $G$  be the group defined in Section \ref{sec:biadditive} and $\psi:A\xhookrightarrow{} \Qlcl ^{\times}$ the fixed injective character of $A\ (\cong Z(\pi_{0}(G))$. Then there exists an unique irreducible representation (upto isomorphism) of $\pi_{0}(G)$  whose central character is $\psi$. The dimension of this irreducible representation is $|\pi_0(K_1)|=|\pi_0(K_2)|$.
\end{corollary}
 \begin{proof} 
 Let us  define a bilinear pairing on $\pi_{0}(G)/A\ (\cong\pi_{0}(K_{1})\times\pi_{0}(K_{2}))$ using $\psi$ as follows: for $\bar{x},\bar{y}\in \pi_{0}(G)/A$,   $\langle \bar{x},\bar{y}\rangle=\psi(xyx^{-1}y^{-1})$. We will prove that $\psi$ is generic, i.e. the pairing $\langle\ ,\ \rangle$ is non-degerate.
 
 Let $\bar{x}=(k_1,k_2)\in \pi_{0}(G)/A=\pi_{0}(K_{1})\times\pi_{0}(K_{2})$ be any non-zero element. Take  $x=(k_{1},k_{2},0)$ and let $y=(b_{1},b_{2},0)$. Then 
 $\langle\bar{x},\bar{y}\rangle=\psi(xyx^{-1}y^{-1})=\psi(B(k_{1},b_{2})-B(b_{1},k_{2}))$.
As $\bar{x}\neq 0$,  without loss of generality assume that $k_{1}\neq 0$. Since  $B$ is a non-degenerate pairing, there exists $b_{2}\in \pi_{0}(K_{2})$ such that $B(k_{1},b_{2})\neq 0$.
 If  $y=(0, b_{2},0)$ then $\langle\bar{x},\bar{y}\rangle=\psi(B(k_{1},b_{2}))\neq 1$.
 This proves that $\psi$ is a generic character of $A$.
 Then the corollary follows from \emph{Stone-Von Neumann} theorem (see \cite{Bump} Ex. 4.1.8). 
\end{proof}

\begin{proof}[\bf{Proof of Corollary \ref{irreducible}}]
Note that action of $G^{\circ}$ on $H_{c}^{2D}(U_{1}\times U_{2}, \mathcal{L}(D))$ is trivial, so the action of $G$ factors through $\pi_{0}(G)$. 
It is easy to see that centre of $\pi_{0}(G)\ (\cong A)$ acts on $H_{c}^{2D}(U_{1}\times U_{2},\mathcal{L}_{\psi}(D))$ by the character $\psi$ and $\pi_{0}(G)$ is two step nilpotent group.
By  \ref{dim eq}, $\dim(H_{c}^{2D}(U_{1}\times U_{2}, \mathcal{L}(D)))$ is $|\pi_{0}(K_{1})|$.
The result now follows from Corollary \ref{generic character}.
\end{proof}

\begin{definition}{(Symmetric biextension.)} 
Let $U\in \cpu$, the biextension $E$ of $U\times U$ is said to be symmetric  if $f=f^{*}$, where $f\ \mbox{and}\ f^{*}$ are  corresponding homomorphisms.
Equivalently, $E$ is symmetric if $\tau^{*}(E)\cong E$. 
We say that a homomorphism $f:U\rightarrow U^{*}$ is symmetric  if the corresponding biextension is symmetric.
\end{definition}
\begin{definition}{(Skew-symmetric biextension.)}
Let $U\in \cpu$, the biextension $E$ of $U\times U$ is said to be skew-symmetric if the restriction of $E$ to the diagonal of $U\times U$ is a trivial $\Qp/\Zp$-torsor. If $E$ is skew-symmetric biextension of $U\times U$  then $f=-f^{*}$, where $f$ and $f^{*}$ are corresponding homomorphisms.
The converse statement holds if $char(k)> 2$.
We say that a homomorphism $f:U\rightarrow U^{*}$ is skew-symmetric  if the corresponding biextension is skew-symmetric.
\end{definition}
\begin{corollary}\label{symmetric} Let  $U$ be a  connected commutative unipotent group over k of dimension $d$.
Let $\mathcal{L}$ be a bimultiplicative $\Qlcl$-local system on $U\times U$ such that the corresponding homomorphism $f:U_{\perf}\rightarrow U_{\perf}^{*}$ is symmetric (resp. skew-symmetric) and let $\dim(\ker f)=\kappa$.  Then the cardinality of $\pi_{0}(\ker f)$ must be $p^{2r}$ for some $r\in\mathbb{Z}$. The relations between two bases of $H_{c}^{2D}(U\times U, \mathcal{L}(D))$ is then given by,
 $$Y^{*}_{b_{1}}=\frac{(-1)^{d-\kappa}}{p^{r}}\sum_{b_{2}\in \pi_{0}(\ker f)}\psi(B(b_{1},b_{2}))X^{*}_{b_{2}}$$
$$X^{*}_{b_{2}}=\frac{(-1)^{d-\kappa}}{p^{r}}\sum_{b_{1}\in \pi_{0}(\ker f)}\psi(-B(b_{1},b_{2}))Y^{*}_{b_{1}}$$
and  where $\psi$ is the fixed character of $\Qp/\Zp$.
\end{corollary}
\begin{proof}
By definition of symmetric biextension $f=f^{*}$ and $\tau^{*}(\L)\cong \L$.
Then $\ker f =\ker f^{*}$ and $m(U,U,\L)=m(U,U,\tau^{*}(\L))$.
Similarly one can show that
 $m(U,U,\L)=m(U,U,\tau^{*}(\L))$ in case of skew-symmetric homomorphism.
 Then the result follows from  the Theorem \ref{MAIN RESULT}.
 \end{proof}

\bibliographystyle{alpha}
\bibliography{references}
\end{document}